\documentclass[reqno]{amsart}
\usepackage[american]{babel}
\usepackage{amssymb,euscript}
\numberwithin{equation}{section}
\newtheorem{theorem}{Theorem}[section]
\newtheorem{lemma}[theorem]{Lemma}

\theoremstyle{definition}
\newtheorem{definition}[theorem]{Definition}
\theoremstyle{remark}
\newtheorem{remark}[theorem]{Remark}

\DeclareMathOperator{\supp}{supp}

\begin{document}

\title[Kawahara--Zakharov--Kuznetsov Equation]
{Initial-Boundary Value Problems on a Half-Strip for the Generalized Kawahara--Zakharov--Kuznetsov Equation}

\author[A.V.~Faminskii]{Andrei~V.~Faminskii}

\thanks{The work was supported by the Ministry of Science and Higher Education of Russian Federation: agreement no. 075-03-2020-223/3 (FSSF-2020-0018).}

\address{Peoples' Friendship University of Russia (RUDN University), 6 Miklukho--Maklaya Street, Moscow, 117198, Russian Federation}

\email{afaminskii@sci.pfu.edu.ru}



\date{}

\begin{abstract}
Initial-boundary value problems on a half-strip with different types of boundary conditions for the generalized Kawahara--Zakharov--Kuznetsov equation with nonlinearity of higher order are considered. In particular, nonlinearity can be quadratic and cubic. Results on global existence and uniqueness in classes of weak and strong solutions and large-time decay of small solutions are established. The solutions are considered in weighted at infinity Sobolev spaces. The use of weighted spaces is crucial for the study. To this end new interpolating inequalities in weighted anisotropic Sobolev spaces are established. Both exponential and power weights are admissible.
\end{abstract}

\maketitle

\section{Introduction. Description of main results}\label{S1}

The famous Korteweg--de~Vries equation (KdV)
$$
u_t +u_{xxx} +bu_x +uu_x=0
$$
is the model equation describing propagation of one dimensional nonlinear waves in dispersive media. As well as the KdV equation the modified Korteweg--de~Vries equation (mKdV)
$$
u_t +u_{xxx} +bu_x +au^2u_x=0
$$
is used in real physical situations (see, for example, \cite{KO}). Moreover, one can consider wave processes with more complicated nonlinear effect and for its description use the generalized KdV equation
$$
u_t +u_{xxx} +bu_x +g'(u)u_x=0.
$$
In the case of dispersive effects of higher orders the KdV equation can be substituted, for example, by the Kawahara equation (see, \cite{Kaw})
$$
u_t -u_{xxxxx} +u_{xxx} +bu_x +u u_x=0
$$
or its generalized analogue
$$
u_t -u_{xxxxx} +u_{xxx} +bu_x +g'(u) u_x=0.
$$

The Zakharov--Kuznetsov equation (ZK) in the case of two spatial dimensions
$$
u_t +u_{xxx}+u_{xyy} +bu_x +uu_x=0
$$
extends the KdV model to the situation, when the  waves propagate in one preassigned ($x$) direction with deformations in the transverse ($y$) direction.
For the first time it was derived in \cite{ZK} for description of  ion-acoustic waves in magnetized plasma. As well as the mKdV equation one can consider the modified Zakharov--Kuznetsov equation
$$
u_t +u_{xxx}+u_{xyy} +bu_x +a u^2u_x=0,
$$
which also has the physical meaning (see \cite{SB}) or more generally the generalized ZK equation
$$
u_t +u_{xxx}+u_{xyy} +bu_x + g'(u)u_x=0.
$$

The Kawahara--Zakharov--Kuznetsov equation (KZK)
$$
u_t -u_{xxxxx} +u_{xxx}+u_{xyy} +bu_x + uu_x=0,
$$
is the natural model for two-dimensional nonlinear waves propagating in the media with the higher order dispersion. Its physical motivation can be found in \cite{EEA}. Its version with more general nonlinearity is written as follows:
$$
u_t -u_{xxxxx} +u_{xxx}+u_{xyy} +bu_x + g'(u)u_x=0
$$
and is the object of the study (in the non-homogeneous case) of the present paper. The goal of the paper is to establish global existence and uniqueness of solutions to initial-boundary value problems for such equation as well as large-time decay under small input data. The obtained results are valid both for the KZK equation itself and for the modified one ($g'(u) = au^2$). 

There is a lot of literature devoted to various aspects of the KdV equation and its extensions with more general nonlinearity and here we do not write about them. The Kawahara equation and its extensions with more general nonlinearity are studied considerably less. The initial value problem and initial-boundary value problems on semi-axes are considered, for example, in \cite{BL97, CK, CG, CDT, DL, FM21, FO14, FO15, GL, KF, L14, LS15, N, SF, TC, ZLSL}. Also during last 30 years there were published several papers devoted to initial-boundary value problems (especially to the pure initial value problem) for the ZK equation and its extensions with more general nonlinearity (see, for example, bibliography in \cite{F18, F20, F21} and recent papers \cite{CFL, KST, MMPP, Sh}). Wide classes of quasilinear dispersive equations were considered, for example, in \cite{F89, F12, Sa}. However, these equations were isotropic with respect to the orders of the highest spatial derivatives (as in the ZK equation).

The two-dimensional Kawahara--Zakharov--Kuznetsov equation was for the first time considered in \cite{L14}. An initial-boundary value problem on a half-strip was studied and results on global existence, uniqueness of regular solutions as well as large-time decay in the case of small initial data were established. The comparison of our results with the results from \cite{L14} are presented further (see Remark \ref{R1.3}). In \cite{LS16} the three-dimensional case of the KZK equation was considered.

In the present paper we consider initial-boundary value problems in a domain $\Pi_T^+=(0,T)\times\Sigma_+$, where $\Sigma_+=\mathbb R_+\times (0,L)=\{(x,y): x>0, 0<y<L\}$ is a half-strip of a given width $L$ and $T>0$ is arbitrary, for an equation
\begin{equation}\label{1.1}
u_t -u_{xxxxx} +u_{xxx}+u_{xyy} +bu_x + g'(u)u_x=f(t,x,y)
\end{equation}
($b\in \mathbb R$ -- constant) with an initial condition
\begin{equation}\label{1.2}
u(0,x,y)=u_0(x,y),\qquad (x,y)\in\Sigma_+,
\end{equation}
a boundary condition 
\begin{equation}\label{1.3}
u(t,0,y)=0,\quad u_x(t,0,y) =0,\qquad (t,y)\in B_{T}=(0,T)\times (0,L),
\end{equation}
and boundary conditions for $(t,x)\in \Omega_{T,+}=(0,T)\times\mathbb R_+$ of one of the following four types: 
\begin{equation}\label{1.4}
\begin{split}
\mbox{whether}\qquad &a)\mbox{ } u(t,x,0)=u(t,x,L)=0,\\
\mbox{or}\qquad &b)\mbox{ } u_y(t,x,0)=u_y(t,x,L)=0,\\
\mbox{or}\qquad &c)\mbox{ } u(t,x,0)=u_y(t,x,L)=0,\\
\mbox{or}\qquad &d)\mbox{ } u \mbox{ is an  $L$-periodic function with respect to $y$.}
\end{split}
\end{equation}
We use the notation "problem \eqref{1.1}--\eqref{1.4}" for each of these four cases.

The choice of such a domain can be verified by the physical nature of the considered model, that is wave propagation in a channel of a finite width from the left boundary. Dirichlet boundary condition a) in \eqref{1.4} from the physical point of view corresponds to wave propagation processes with the absence of deformations on the boundaries of the channel, Neumann boundary condition b) is the no-flow condition through this boundaries, periodic boundary condition d) describes wave propagation in the media with periodic structure. Note that results on existence and uniqueness are the same for all types of boundary conditions \eqref{1.4}.

Results on global existence are bases on estimates which are the analogues of the following conservation laws for the initial value problem
\begin{equation}\label{1.5}
\iint_{\mathbb R^2} u^2\,dxdy = \text{const}, \quad
\iint_{\mathbb R^2} \bigl(u_{xx}^2 +u_x^2 +u_y^2 - 2g^*(u)\bigr)\,dxdy =\text{const}, 
\end{equation}
where here and further 
\begin{equation}\label{1.6}
g^*(u) \equiv \int_0^u g(\theta)\, d\theta.
\end{equation}
Further we differ the situations when only the analogue of the first conservation law \eqref{1.5} is used and when all two of them are used and call solutions from the first class weak solutions, while from the second class -- strong solutions. Without loss of generality we assume further that $g(0)=0$.

The results are obtained in the weighted at $+\infty$ with respect to $x$ anisotropic Sobolev spaces. The use of weighted spaces is crucial for the developed theory. Both power and exponential weights are allowed. Special interpolating multiplicative inequalities are established further (see Lemma \ref{L2.1})

In what follows (unless stated otherwise) $i$, $j$, $k$, $l$, $m$, $n$ mean non-negative integers, $p\in [1,+\infty]$, $s\in\mathbb R$.  For any multi-index $\nu=(\nu_1,\nu_2)$ let $\partial^\nu =\partial^{\nu_1}_{x}\partial^{\nu_2}_{y}$. Let $L_{p,+}=L_p(\Sigma_+)$.

Introduce special function spaces $\widetilde H_+^k$ taking into account boundary conditions \eqref{1.4} and the anisotropic character of the considered equation. Let 
$\widetilde H^0_+=L_{2,+}$; for $k \geq 1$ let the space $\widetilde H_+^k$ consists of functions $\varphi(x)$ such that $\partial^\nu \varphi \in L_{2,+}$ if $\nu_1 + 2\nu_2 \leq 2k$ and \linebreak in the case a) 
$$
\partial_y^{2m}\varphi|_{y=0}=\partial_y^{2m}\varphi|_{y=L}=0 \quad \forall m<k/2,
$$
in the case b) 
$$
\partial_y^{2m+1}\varphi|_{y=0}=\partial_y^{2m+1}\varphi|_{y=L}=0\quad \forall m<(k-1)/2,
$$
in the case c) 
$$
\partial_y^{2m}\varphi|_{y=0}=0 \quad \forall m<k/2,\quad \partial_y^{2m+1}\varphi|_{y=L}=0\quad \forall m<(k-1)/2,
$$ 
in the case d) 
$$
\partial_y^{m}\varphi|_{y=0}=\partial_y^{m}\varphi|_{y=L} \quad \forall m<k.
$$  

We say that $\psi(x)$ is an admissible weight function if $\psi$ is an infinitely smooth positive function on $\overline{\mathbb R}_+$, such that 
\begin{gather}\label{1.7}
|\psi^{(j)}(x)|\leq c(j)\psi(x)\quad \text{for each natural\ } j \text{\ and\ } \forall x\geq 0;\\
\label{1.8}
\psi(x_1) \leq c\psi(x_2)\quad \text{for certain constant\ }c>0 \quad \text{and } \forall\  x_1 +1 >x_2 > x_1 \geq 0.
\end{gather}

Note that such a function satisfies an inequality $\psi(x) \leq ce^{c_0 x}$ for certain positive constants $c_0$, $c$ and all $x\geq 0$. It was shown in \cite{F12} that $\psi^s$ for any $s\in\mathbb R$ satisfies \eqref{1.7}; it is easy to see that $\psi^s$ also satisfies \eqref{1.8}, so $\psi^s$ is an admissible weight function. Any exponent $e^{2\alpha x}$ as well as $(1+x)^{2\alpha}$ are admissible weight functions. As an another important example of admissible functions, we define 
$\rho_0(x)\equiv 1+ \frac{2}{\pi}\arctan x$. Note that both $\rho_0$ and $\rho'_0$ are admissible weight functions.

Note that, for example, in \cite{F18, F20, F21} condition \eqref{1.8} was not used for the definition of the admissible weight function. Here we have to introduce it in order to establish interpolating inequalities in anisotropic Sobolev spaces in Lemma \ref{L2.1}.

For an admissible weight function $\psi(x)$ let $\widetilde H^{k,\psi(x)}_+$ be a space of functions $\varphi(x,y)$ such that $\varphi\psi^{1/2}(x)\in \widetilde H^k_+$. Let $L_{2,+}^{\psi(x)}= \widetilde H^{0,\psi(x)}_+= \{\varphi(x,y): \varphi\psi^{1/2}(x)\in L_{2,+}\}$. Obviously, $L_{2,+}^{\rho_0(x)}=L_{2,+}$.

We construct solutions to the considered problems in spaces $X_w^{k,\psi(x)}(\Pi_T^+)$ for $k=0$ (weak solutions) and $k=1$ (strong solutions), for admissible weight functions $\psi(x)$ verifying $\psi'(x)$ are also admissible weight functions, consisting of functions $u(t,x,y)$, such that 
\begin{equation}\label{1.9}
u\in C_w([0,T]; \widetilde H_+^{k,\psi(x)})\cap L_2(0,T;\widetilde H_+^{k+1,\psi'(x)})
\end{equation}
(the subscript $w$ means the weak continuity). Let $X_w^{\psi(x)}(\Pi_T^+)=X_w^{0,\psi(x)}(\Pi_T^+)$. For auxiliary linear results we also use spaces $X^{k,\psi(x)}(\Pi_T^+)$, where in comparison with $X_w^{k,\psi(x)}(\Pi_T^+)$ the weak continuity with respect to $t$ in \eqref{1.9} is substituted by the strong one. 

Define also
\begin{equation}\label{1.10}
\lambda^+(u;T) =
\sup_{x_0\geq 0}\int_0^T\! \int_{x_0}^{x_0+1}\! \int_0^L u^2\,dydxdt.
\end{equation}

Introduce the notion of weak solutions to the considered problems. To this end define special function spaces of smooth functions. Let 
$\widetilde{\EuScript S}(\overline{\Sigma}_+)$ be a space of infinitely smooth on $\overline{\Sigma}_+$ functions $\varphi(x,y)$ such that $\displaystyle{(1+x)^n|\partial^\alpha\varphi(x,y)|\leq c(n,\alpha)}$ for any $n$, multi-index $\alpha$, $(x,y)\in \overline{\Sigma}_+$ and $\partial_y^{2m}\varphi\big|_{y=0} =\partial_y^{2m}\varphi\big|_{y=L}=0$ in the case a), $\partial_y^{2m+1}\varphi\big|_{y=0} =\partial_y^{2m+1}\varphi\big|_{y=L}=0$ in the case b), $\partial_y^{2m}\varphi\big|_{y=0} =\partial_y^{2m+1}\varphi\big|_{y=L}=0$ in the case c), $\partial_y^{m}\varphi\big|_{y=0} =\partial_y^{m}\varphi\big|_{y=L}$ in the case d) for any $m$.

\begin{definition}\label{D1.1}
Let $u_0\in L_{2,+}$, $f\in L_1(0,T;L_{2,+})$. A function $u\in L_2(\Pi_T^+)$ is called a weak solution to problem \eqref{1.1}--\eqref{1.4} if for any function $\phi\in \widetilde{\EuScript S}(\overline{\Sigma}_+)$, $\phi\big|_{t=T}=0$, $\phi\big|_{x=0} =\phi_x\big|_{x=0}= \phi_{xx}\big|_{x=0} =0$, the function $g(u)\phi_x\in L_1(\Pi_T^+)$ and the following equality holds:
\begin{equation}\label{1.11}
\iiint_{\Pi_T^+}\Bigl[u(\phi_t -\phi_{xxxxx} +\phi_{xxx}+\phi_{xyy} +b\phi_x) +g(u)\phi_x + f\phi \Bigl]\,dxdydt 
+\iint_{\Sigma_+} u_0\phi\big|_{t=0}\,dxdy =0.
\end{equation}
\end{definition}

Now we can formulate the main results of the paper.

\begin{theorem}[global existence and uniqueness of weak solutions]\label{T1.1}
Let $u_0\in L_{2,+}^{\psi(x)}$, $f\in L_1(0,T;  L_{2,+}^{\psi(x)})$ for certain admissible weight function $\psi(x)$, such that $\psi'(x)$ is also an admissible weight function. Let the function $g\in C^1(\mathbb R)$ and for certain constants $p\in [0,8/3)$ and $c>0$
\begin{equation}\label{1.12}
|g'(u)| \leq c |u|^p\quad \forall u\in \mathbb R,
\end{equation}
and if $p>1$ the function $\psi$ for certain constants $n$ and $c>0$ satisfies an inequality $\psi(x) \leq c(1+x)^n\psi'(x)$.
Then there exists a weak solution to problem \eqref{1.1}--\eqref{1.4} $u\in X_w^{\psi(x)}(\Pi_T^+)$; moreover $\lambda^+(u_{xx};T) + \lambda^+(u_y;T) <+\infty$.  If, in addition, $p\leq 2$ in \eqref{1.12} and for certain positive constant $c_0$
\begin{equation}\label{1.13}
(\psi'(x))^{2+3p} \psi^{p-2}(x)\geq c_0\quad \forall x\geq 0,
\end{equation}
then this solution is unique in $X_w^{\psi(x)}(\Pi_T^+)$.
\end{theorem}

\begin{remark}\label{R1.1}
The exponential weight $\psi(x)\equiv e^{2\alpha x}$ $\forall \alpha>0$ and the power weight $\psi(x)\equiv (1+x)^{2\alpha}$, $\alpha\geq (3p+2)/(8p)$ satisfy the hypothesis of the theorem (including uniqueness). If $u_0\in L_{2,+}$, $f\in L_1(0,T;L_{2,+})$, there exists a weak solution $u\in C_w([0,T];L_{2,+})$, $\lambda^+(u_{xx};T) + \lambda^+(u_y;T) <+\infty$. In comparison with the results from \cite{F18} for the similar problems for the ZK equation itself ($p=1$), the uniqueness assumption for the power weights $\alpha\geq 1$ in \cite{F18} is weakened here to the assumption $\alpha\geq 5/8$.
\end{remark}

\begin{theorem}[global existence and uniqueness of strong solutions]\label{T1.2}
Let $u_0\in \widetilde H^{1,\psi(x)}_+$, $f\in L_2(0,T;  \widetilde H^{1,\psi(x)}_+)$ for certain admissible weight function $\psi(x)$, such that $\psi'(x)$ is also an admissible weight function, $u_0(0,y)\equiv 0$. Let $g\in C^2(\mathbb R)$ and verifies condition \eqref{1.12} for $p\in [0,8/3)$.
Then there exists a strong solution to problem \eqref{1.1}--\eqref{1.4} $u\in X_w^{1,\psi(x)}(\Pi_T^+)$; moreover, $\lambda^+(u_{xxxx};T) + \lambda^+(u_{yy};T)<+\infty$. 
If, in addition, for certain constants $q\geq 0$ and $c>0$
\begin{equation}\label{1.14}
|g''(u)| \leq c |u|^q\quad \forall u\in\mathbb R
\end{equation}
and for certain positive constant $c_0$
\begin{equation}\label{1.15}
\psi'(x) \psi^{4q+3}(x) \geq c_0\quad \forall x\geq 0,
\end{equation}
then this solution is unique in the space $X_w^{1,\psi(x)}(\Pi_T^+)$.
\end{theorem}

\begin{remark}\label{R1.2}
Let $|g'(u)|\leq c|u|^p$ for certain $p\in [1,8/3)$, $|g''(u)|\leq c |u|^{p-1}$. Then the exponential weight $\psi(x)\equiv e^{2\alpha x}$ $\forall \alpha>0$ and the power weight $\psi(x)\equiv (1+x)^{2\alpha}$ if $\alpha\geq 1/(8p)$ satisfy the hypothesis of the theorem. If $u_0\in \widetilde H_+^1$, $u_0(0,y)\equiv 0$ and $f\in L_2(0,T;\widetilde H_+^1)$ there exists a weak solution $u\in C_w([0,T];\widetilde H_+^1)$, $\lambda^+(u_{xxxx};T) + \lambda^+(u_{yy};T)<+\infty$.  In comparison with the results from \cite{F18} for the similar problems for the ZK equation itself, the uniqueness assumption for the power weights $\alpha\geq 1/2$ in \cite{F18} is weakened here to the assumption $\alpha\geq 1/8$.  In comparison with the results from \cite{F21} for the similar problems for the modified ZK equation itself $(p=2)$, the uniqueness assumption for the power weights $\alpha\geq 3/8$ in \cite{F18} is weakened here to the assumption $\alpha\geq 1/16$. Moreover, in \cite{F21} the global existence result for the modified ZK equation in the focusing case $g'(u) = u^2$ was established only for small input data.
\end{remark}

Finally, we present results on large-time decay of solutions for the types a) and c) of boundary conditions \eqref{1.4}. 

\begin{theorem}[large-time decay of weak solutions]\label{T1.3}
Let the function $g\in C^1(\mathbb R)$ satisfies inequality \eqref{1.12} for $p\in (0,2]$.
Let $L_0=+\infty$ if $b\leq 0$, and if $b>0$ there exists $L_0>0$, such that in both cases for any $L\in (0,L_0)$ there exist $\alpha_0>0$, $\epsilon_0>0$ and $\beta>0$, such that if $u_0\in L_{2,+}^{e^{2\alpha x}}$ for $\alpha\in (0,\alpha_0]$, $\|u_0\|_{L_{2,+}} \leq \epsilon_0$, $f\equiv 0$, the corresponding unique solution $u(t,x,y$) to problem \eqref{1.1}--\eqref{1.4} in the cases a) and c) in \eqref{1.4} from the space $X_w^{e^{2\alpha x}}(\Pi_T^+)$ $\forall T>0$ satisfies an inequality
\begin{equation}\label{1.16}
\|e^{\alpha x}u(t,\cdot,\cdot)\|^2_{L_{2,+}}\leq e^{-\alpha\beta t}\|e^{\alpha x}u_0\|_{L_{2,+}}^2\qquad \forall t\geq 0.
\end{equation}
\end{theorem}

\begin{theorem}[large-time decay of strong solutions]\label{T1.4}
Let the function $g\in C^2(\mathbb R)$ satisfies inequality \eqref{1.12} for $p\in [1,8/3)$ and inequality \eqref{1.14} for $q=p-1$.
Let $L_0=+\infty$ if $b\leq 0$, and if $b>0$ there exists $L_0>0$, such that in both cases for any $L\in (0,L_0)$ there exist $\alpha_0>0$, $\epsilon_0>0$ and $\beta>0$, such that if $u_0\in \widetilde H_+^{1,e^{2\alpha x}}$ for $\alpha\in (0,\alpha_0]$, $\|u_0\|_{L_{2,+}} \leq \epsilon_0$, $u_0(0,y)\equiv 0$, $f\equiv 0$, the corresponding unique solution $u(t,x,y$) to problem \eqref{1.1}--\eqref{1.4} in the cases a) and c) in \eqref{1.4} from the space $X_w^{1,e^{2\alpha x}}(\Pi_T^+)$ $\forall T>0$ satisfies an inequality
\begin{equation}\label{1.17}
\|e^{\alpha x} u(t,\cdot,\cdot)\|^2_{\widetilde H_+^1} \leq c e^{-\alpha\beta t}\qquad \forall t\geq 0,
\end{equation}
where the constant $c$ depends on $b$, $\alpha$, $\beta$, $\|u_0\|_{\widetilde H_+^{1,e^{2\alpha x}}}$ and the properties of the function $g$.
\end{theorem}

\begin{remark}\label{R1.3}
Now we can describe in more details the results from \cite{L14}. There (in the notation of the present paper) the initial-boundary value problem for the homogeneous Kawahara--Zakharov--Kuznetsov equation with $b=1$ was considered on the half-strip $\Sigma_+$ with initial and boundary conditions \eqref{1.2}, \eqref{1.3} and \eqref{1.4} in the case a). The weighted at $+\infty$ spaces with the exponential weights $\psi(x) \equiv e^{2\alpha x}$, $0< \alpha <\sqrt{3/20}$, were used. It was assumed that $u_0$, $u_{0xx}$, $u_{0yy}$, $\Delta u_{0x} - \partial_x^5 u_0$ lied in the space $L_{2,+}^{\psi(x)}$. Then existence and uniqueness of global regular solutions (more regular than the ones, constructed in Theorems \ref{T1.1} and \ref{T1.2}) were proved. Results on large-time exponential decay of small solutions were also established.
\end{remark}

The paper is organized as follows. Section~\ref{S2} consists of certain preliminaries concerning corresponding linear results and interpolating inequalities. Section~\ref{S3} is devoted to existence. Results on uniqueness and continuous dependence are proved in Section~\ref{S4}. Decay of solutions is studied in Section~\ref{S5}. The proofs of certain auxiliary interpolating inequalities are shifted to Appendix~\ref{A1}.

\section{Preliminaries}\label{S2}

From now on we drop limits of integration in integrals with respect to $x$ and $y$ over the whole half-strip $\Sigma_+$ and with respect to $x$ over the half-line $\mathbb R_+$.

Introduce certain additional notation.

Let $\eta(x)$ denotes a cut-off function, namely, $\eta$ is an infinitely smooth non-decreasing function on $\mathbb R$ such that $\eta(x)=0$ for $x\leq 0$, $\eta(x)=1$ for $x\geq 1$, $\eta(x)+\eta(1-x) \equiv 1$.

Let $\EuScript S_{exp}(\overline{\Sigma}_+)$ denotes a space of infinitely smooth functions $\varphi(x,y)$ on $\overline{\Sigma}_+$,  such that $e^{n x}|\partial^\nu\varphi(x,y)|\leq c(n,\nu)$ for any $n$, multi-index $\nu$, $(x,y)\in \overline{\Sigma}_+$.

Let $\widetilde{\EuScript S}_{exp}(\overline{\Sigma}_+)$ be a subspace of $\EuScript S_{exp}(\overline{\Sigma}_+)$, consisting of functions, on the boundaries $y=0$, $y=L$ verifying the same conditions as in the definition of the space $\widetilde{\EuScript S}(\overline{\Sigma}_+)$. The spaces $\widetilde{\EuScript S}_{exp}(\overline{\Sigma}_+)$ are obviously dense in $\widetilde H_+^k$ for any $k$.

The following interpolating inequality for weighted anisotropic Sobolev spaces is crucial for the study.

\begin{lemma}\label{L2.1}
Let $\psi_1(x)$, $\psi_2(x)$ be two admissible weight functions, either $m=0$, $q\in [2,+\infty]$ or $m=1$, $q\in [2,6]$,
\begin{equation}\label{2.1}
s = s(m,q) = \frac{2m+3}8 - \frac3{4q}.
\end{equation}
Then for every function $\varphi(x,y)$ satisfying $\bigl(|\varphi_{xx}| + |\varphi_y| +|\varphi|\bigr)\psi_1^{1/2}(x)\in L_{2,+}$, $\varphi\psi_2^{1/2}(x)\in L_{2,+}$, $\varphi(0,y)\equiv 0$, the following inequality holds:
\begin{equation}\label{2.2}
\bigl\| \partial_x^m \varphi\psi_1^s\psi_2^{1/2-s}\bigr\|_{L_{q,+}} \leq c
\bigl\|\bigl(|\varphi_{xx}| + |\varphi_y| +|\varphi|\bigr)\psi_1^{1/2}\bigr\|^{2s}_{L_{2,+}}
\bigl\|\varphi\psi_2^{1/2}\bigr\|_{L_{2,+}}^{1-2s}, 
\end{equation}
where the constant $c$ depends on $L$, $m$, $q$ and the properties of the functions $\psi_j$; if, in addition, $\varphi\big|_{y=0}=0$ or $\varphi\big|_{y=L}=0$ then this constant is uniform with respect to $L$.
\end{lemma}

\begin{proof}
Without loss of generality, assume that $\varphi$ is a smooth decaying at $+\infty$ function (for example, $\varphi\in \EuScript S_{exp}(\overline{\Sigma}_+)$).

The proof is based on two interpolating inequalities for anisotropic Sobolev spaces following from \cite[Theorem~10.1]{BIN}. If $Q_n =(n,n+1) \times (0,L)$, then
\begin{equation}\label{2.3}
\|f\|_{L_\infty(Q_n)} \leq c(L)\Bigl(\iint_{Q_n} \bigl(f^2_{xx} +f_y^2 +f^2\bigr)\, dxdy\Bigr)^{3/8} \Bigl(\iint_{Q_n} f^2 \,dxdy\Bigr)^{1/8};
\end{equation}
and
\begin{equation}\label{2.4}
\|f_x\|_{L_6(\Sigma_+)} \leq c(L) \Bigl(\iint \bigl(f^2_{xx} +f_y^2 +f^2\bigr)\, dxdy\Bigr)^{1/2}.
\end{equation}

Note that property \eqref{1.8} implies that
$$
\sup\limits_{x\in (n,n+1)} \psi_j(x) \leq c \inf\limits_{x\in (n,n+1)} \psi_j(x).
$$
Then \eqref{2.3} yields that
\begin{multline*}
\sup\limits_{(x,y)\in Q_n}|\varphi| \psi_1^{3/8} \psi_2^{1/8} \leq c\inf\limits_{x\in (n,n+1)} \psi_1^{3/8} \inf\limits_{x\in (n,n+1)}\psi_2^{1/8} \cdot
\|\varphi\|_{L_\infty(Q_n)}\\  \leq
c(L) \Bigl(\iint_{Q_n} \bigl(\varphi^2_{xx} +\varphi_y^2 +\varphi^2\bigr)\psi_1\, dxdy\Bigr)^{3/8} \Bigl(\iint_{Q_n} \varphi^2\psi_2 \,dxdy\Bigr)^{1/8},
\end{multline*}
whence \eqref{2.2} in the case $m=0$, $q=+\infty$ follows. As a result, for $q\in (2,+\infty)$, $s=s(0,q)$
\begin{multline*}
\bigl\| \varphi\psi_1^s\psi_2^{1/2-s}\bigr\|_{L_{q,+}} \leq \left(\bigl\|\varphi\psi_1^{3/8}\psi_2^{1/8}\bigr\|_{L_{\infty,+}}\right)^{(q-2)/q}
\Bigl(\iint \varphi^2\psi_2\,dxdy\Bigr)^{1/q} \\ \leq
c(L) \Bigl(\iint \bigl(\varphi^2_{xx} +\varphi_y^2 +\varphi^2\bigr)\psi_1\, dxdy\Bigr)^{s} \Bigl(\iint \varphi^2\psi_2 \,dxdy\Bigr)^{1/2-s},
\end{multline*}
since
$$
s= \frac38(1-\frac2q),\quad \frac12 -s = \frac18(1-\frac2q) +\frac1q.
$$

Next, integration by parts yields an equality (here the assumption $\varphi(0,y)\equiv 0$ is used)
$$
\iint \varphi_x^2 \psi_1^{1/2}\psi_2^{1/2}\, dxdy =
-\iint \varphi_{xx}\psi_1^{1/2}\cdot \varphi\psi_2^{1/2}\,dxdy 
-\iint\varphi\varphi_x(\psi_1^{1/2}\psi_2^{1/2})'\,dxdy.
$$
Here property \eqref{1.7} implies that
\begin{multline*}
\Bigl|\iint\varphi\varphi_x(\psi_1^{1/2}\psi_2^{1/2})'\,dxdy\Bigr|  \\ \leq c \Bigl(\iint \varphi_x^2 \psi_1^{1/2}\psi_2^{1/2}\, dxdy \Bigr)^{1/2}  
\Bigl(\iint \varphi^2 \psi_1\, dxdy \Bigr)^{1/4}\Bigl(\iint \varphi^2 \psi_2\, dxdy \Bigr)^{1/4},
\end{multline*} 
whence \eqref{2.2} in the case $m=1$, $q=2$ follows.

In the case $m=1$, $q=6$ estimate \eqref{2.4} and property\eqref{1.7} yield that (here $s(1,6) = 1/2$)
\begin{multline*}
\bigl\|\varphi_x\psi_1^{1/2}\bigr\|_{L_{6,+}} \leq \bigl\| \partial_x (\varphi\psi_1^{1/2})\bigr\|_{L_{6,+}} + 
c\bigl\| \varphi\psi_1^{1/2}\bigr\|_{L_{6,+}} \\ \leq
c(L) \Bigl(\iint \bigl(\varphi^2_{xx} + \varphi_x^2 +\varphi_y^2 +\varphi^2\bigr)\psi_1\, dxdy\Bigr)^{1/2} + 
c\bigl\| \varphi\psi_1^{1/2}\bigr\|_{L_{6,+}},
\end{multline*}
and the already obtained estimate \eqref{2.2} for $m=1$, $q=2$ and $m=0$, $q=6$ in the case $\psi_2\equiv \psi_1$ provide the desired result. In turn, if $m=1$, $q\in (2,6)$, $s=s(1,q)$
\begin{multline*}
\bigl\| \varphi_x\psi_1^s\psi_2^{1/2-s}\bigr\|_{L_{q,+}}  = 
\Bigl(\iint \bigl(|\varphi_x|\psi_1^{1/2}\bigr)^{3(q-2)/2}\cdot \bigl(|\varphi_x|(\psi_1\psi_2)^{1/4}\bigr)^{(6-q)/2}\,dxdy \Bigr)^{1/q} \\ \leq
\Bigl(\iint \varphi_x^6 \psi_1^3\,dxdy\Bigr)^{(q-2)/(4q)} \Bigl(\iint \varphi_x^2 \psi_1^{1/2} \psi_2^{1/2} \,dxdy\Bigr)^{(6-q)/(4q)},
\end{multline*}
and the already obtained estimate \eqref{2.2} for $m=1$, $q=6$ and $q=2$ provide the desired result, since
$$
s= 6s(1,6) \frac{q-2}{4q} +2s(1,2) \frac{6-q}{4q},\quad \frac12 -  s = \bigl(1 -2 s(1,2)\bigr) \frac{6-q}{4q}.
$$

Finally, if for example $\varphi\big|_{y=L}=0$, extend the function $\varphi$ by zero to the quarter-plane $\mathbb R_+ \times \mathbb R_+$ and carry out the same argument with the use of \eqref{2.3} for $(n,n+1)\times \mathbb R_+$ and \eqref{2.4} for the quarter-plane, then estimate \eqref{2.2} becomes uniform with respect to $L$.
\end{proof}

In the following lemma the assumption $\varphi(0,y)\equiv 0$ is not used.

\begin{lemma}\label{L2.2}
Let $\psi(x)$ be an admissible weight function, then there exists a constant $c$, depending on the properties of the function $\psi$, such that for any function $\varphi(x,y)$, verifying $\varphi_{xx}, \varphi \in L_{2,+}^{\psi(x)}$, the following inequalities hold:
\begin{gather}\label{2.5}
\iint \varphi_x^2\psi \,dxdy \leq c \Bigl(\iint \varphi_{xx}^2 \psi\, dxdy\Bigr)^{1/2} \Bigl(\iint \varphi^2 \psi\, dxdy\Bigr)^{1/2} +
c\iint \varphi^2 \psi\, dxdy, \\
\label{2.6}
\int_0^L \varphi_x^2\big|_{x=0}\,dy \leq c \Bigl(\iint \varphi_{xx}^2\psi\,dxdy\Bigr)^{3/4} \Bigl(\iint \varphi^2\psi\,dxdy\Bigr)^{1/4}+
c\iint \varphi^2\psi\,dxdy.
\end{gather}
\end{lemma}

\begin{proof}
The proof is based on an elementary inequality
$$
\sup\limits_{x\geq 0} f^2(x) \leq 2\Bigl(\int (f')^2\,dx\Bigr)^{1/2} \Bigl(\int f^2\,dx\Bigr)^{1/2}
$$
and an equality
$$
\int (f')^2\psi \,dx = -(f' f\psi)\big|_{0} - \int f'' f \psi\, dx - \int f' f \psi'\, dx.
$$
Then
\begin{multline*}
\int_0^L |\varphi_x\varphi\psi|\big|_{x=0} \,dy \leq 2\psi(0)\Bigl( \int_0^L \sup\limits_{x\geq 0} \varphi_x^2 \,dy\Bigr)^{1/2} 
\Bigl( \int_0^L \sup\limits_{x\geq 0} \varphi^2 \,dy\Bigr)^{1/2}  \\ \leq
c \Bigl(\iint \varphi_{xx}^2\psi\,dxdy\Bigr)^{1/4} \Bigl(\iint \varphi_{x}^2\psi\,dxdy\Bigr)^{1/2} \Bigl(\iint \varphi^2\psi\,dxdy\Bigr)^{1/4}
\end{multline*}
and as a result with the use of \eqref{1.7}
\begin{multline*}
\iint \varphi_x^2\psi \,dxdy   \leq c \Bigl(\iint \varphi_{xx}^2\psi\,dxdy\Bigr)^{1/4} \Bigl(\iint \varphi_{x}^2\psi\,dxdy\Bigr)^{1/2} \Bigl(\iint \varphi^2\psi\,dxdy\Bigr)^{1/4} \\+
\Bigl(\iint \varphi_{xx}^2\psi\,dxdy\Bigr)^{1/2} \Bigl(\iint \varphi^2\psi\,dxdy\Bigr)^{1/2} + 
c\Bigl(\iint \varphi_{x}^2\psi\,dxdy\Bigr)^{1/2} \Bigl(\iint \varphi^2\psi\,dxdy\Bigr)^{1/2},
\end{multline*}
whence \eqref{2.5} follows. Finally, 
$$
\int_0^L \varphi_x^2\big|_{x=0}\,dy \leq c \Bigl(\iint \varphi_{xx}^2\psi\,dxdy\Bigr)^{1/2} \Bigl(\iint \varphi_x^2\psi\,dxdy\Bigr)^{1/2}
$$
and \eqref{2.6} follows from \eqref{2.5}.
\end{proof}

Introduce also anisotropic Sobolev spaces with smoothness properties only with respect to $x$. Let $H_+^{(k,0)}$ be a space of functions $\varphi(x,y)\in L_{2,+}$ such that $\partial_x^j \varphi\in L_{2,+}$ for $j\leq k$ endowed with the natural norm $\|\varphi\|_{H_+^{(k,0)}} = \Bigl(\sum\limits_{j=0}^k \|\partial_x^j \varphi\|^2_{L_{2,+}}\Bigr)^{1/2}$. Let $H_+^{(-m,0)} =\bigl\{\varphi(x,y) = \sum\limits_{j=0}^m \varphi_j(x,y): \forall\varphi_j \in L_{2,+}\bigr\}$, endowed with the natural norm $\|\varphi\|_{H_+^{(-m,0)}} = \Bigl(\sum\limits_{j=0}^m \|\varphi_j\|^2_{L_{2,+}}\Bigr)^{1/2}$.

\begin{lemma}\label{L2.3}
If $\varphi\in H_+^{(k,0)}$, $\partial_x^{n}\varphi \in H_+^{(-m,0)}$ for $n\geq k+m$, then $\partial_x^{k+1}\varphi \in L_{2,+}$ and for certain constant $c=c(k,m,n)$
\begin{equation}\label{2.7}
\|\partial_x^{k+1} \varphi\|_{L_{2,+}} \leq c\bigl(\|\partial_x^{n}\varphi\|_{H_+^{(-m,0)}} + \|\varphi\|_{H_+^{(k,0)}}\bigr).
\end{equation}
\end{lemma}

\begin{proof}
Since $\varphi(\cdot,y)\in H^k(\mathbb R_+)$, $\partial_x^{n}\varphi(\cdot,y)\in H^{-m}(\mathbb R_+)$ for a.e. $y\in (0,L)$ the assertion of the lemma follows from 
Lemma \ref{LA.3} from the Appendix.
\end{proof}

For the large-time decay results we need the Steklov inequalities in the following form: 
\begin{equation}\label{2.8}
\int_0^L f^2(y)\,dy \leq \frac{\varkappa L^2}{\pi^2} \int_0^L \bigl(f'(y)\bigr)^2\,dy,
\end{equation}
where $\varkappa=1$ if $f\in H_0^1(0,L)$, $\varkappa=4$ if $f\in H^1(0,L)$, $f\big|_{y=0}=0$. For $f\in H_0^1(0,L)$ \eqref{2.8} is the classic inequality (see, for example \cite{L14}), in the second case it can be easily derived from \eqref{2.8} for the first case by the odd extension of the function $f$ through the point $y=L$.

Let $\psi_l(y)$, $l=1,2\dots$, be the orthonormal in $L_2(0,L)$ system of the eigenfunctions for the operator $(-\psi'')$ on the segment $[0,L]$ with corresponding boundary conditions   $\psi(0)=\psi(L)=0$ in the case a), $\psi'(0)=\psi'(L)=0$ in the case b), $\psi(0)=\psi'(L)=0$ in the case c), $\psi(0)=\psi(L),\psi'(0)=\psi'(L)$ in the case d), $\lambda_l$ be the corresponding eigenvalues. Such systems are well-known and can be written in trigonometric functions. We hope that from the context it will be clear further when we use this eigenfunctions $\psi_l(y)$ and when admissible weight functions $\psi(x)$.

\bigskip
Besides equation \eqref{1.1} we consider a linear equation
\begin{equation}\label{2.9}
u_t-u_{xxxxx}+u_{xxx}+u_{xyy}+bu_x=f(t,x,y)
\end{equation}
with initial and boundary conditions \eqref{1.2}--\eqref{1.4}. Weak solutions to this problem are understood similarly to Definition \ref{D1.1}. First establish a result on well-posedness in the class of infinitely smooth functions.

\begin{lemma}\label{L2.4}
Let $u_0\in \widetilde{\EuScript S}_{exp}(\overline{\Sigma}_+)$, $f\in C^\infty\bigl([0,T]; \widetilde{\EuScript S}_{exp}(\overline{\Sigma}_+)\bigr)$. Set $\widetilde \Phi_0(x,y) \equiv u_0(x,y)$ and for $j\geq 1$
\begin{equation}\label{2.10}
\widetilde\Phi_j(x,y) \equiv \partial^{j-1}_t f(0,x,y) + (\partial_x^5 -\partial_x^3 -\partial_x\partial_y^2 -b\partial_x)\widetilde\Phi_{j-1}(x,y) 
\end{equation}
and let $\widetilde\Phi_j(0,y) = \widetilde \Phi_{jx}(0,y)\equiv 0$ for any $j$.
Then there exists a unique solution to problem \eqref{2.9}, \eqref{1.2}--\eqref{1.4} $u\in C^\infty\bigl([0,T]; \widetilde{\EuScript S}_{exp}(\overline{\Sigma}_+)\bigr)$.
\end{lemma}

\begin{proof}
First of all consider the corresponding initial value problem. Let $\Sigma= \mathbb R\times (0,L)$ and 
$\widetilde{\EuScript S}(\overline{\Sigma})$ be a space of infinitely smooth on $\overline{\Sigma}$ functions $\varphi(x,y)$ such that $\displaystyle{(1+|x|)^n|\partial^\alpha\varphi(x,y)|\leq c(n,\alpha)}$ for any $n$, multi-index $\alpha$, $(x,y)\in \overline{\Sigma}$ and on the boundaries $y=0$, $y=L$ verifying the same conditions as in the definition if the space $\widetilde{\EuScript S}(\overline{\Sigma}_+)$. Extend the functions $u_0$ and $f$ to the whole strip such that $u_0\in \widetilde{\EuScript S}(\overline{\Sigma})$, $f\in C([0,T]; \widetilde{\EuScript S}(\overline{\Sigma}))$ and consider problem \eqref{2.9} (in $\Pi_T =(0,T)\times\Sigma$), \eqref{1.2} (in $\Sigma$), \eqref{1.4} (in $\Omega_{T}=(0,T)\times\mathbb R$). Then with the use of the Fourier transform for the variable $x$ and the Fourier series for the variable $y$  a solution to  problem \eqref{2.9}, \eqref{1.2}, \eqref{1.4} can be written as follows:
$$
u(t,x,y)=\frac{1}{2\pi}\int_{\mathbb R}\ \sum_{l=1}^{+\infty}e^{i\xi x}\psi_l(y)\widehat{u}(t,\xi,l)\,d\xi,
$$
where
\begin{equation}\label{2.11}
\widehat{u}(t,\xi,l)=\widehat{u_0}(\xi,l)e^{i(\xi^5 +\xi^3+\xi\lambda_l -b\xi)t} +
\int_0^t\widehat{f}(\tau,\xi,l)e^{i(\xi^5 +\xi^3+\xi\lambda_l -b\xi)(t-\tau)}\,d\tau,
\end{equation}
$$
\widehat{u_0}(\xi,l)\equiv\iint_{\Sigma} e^{-i\xi x}\psi_l(y)u_0(x,y)\,dxdy,
$$
$$
\widehat f(t,\xi,l)\equiv\iint_{\Sigma} e^{-i\xi x}\psi_l(y) f(t,x,y)\,dxdy.
$$
According to the properties of the functions $u_0$ and $f$ this solution $u\in C^\infty([0,T],\widetilde{\EuScript S}(\overline{\Sigma}))$.

Next, let $v\equiv \partial^k_x\partial^l_y u$ for some $k, l$. Then the function $v$ satisfies an equation of \eqref{2.9} type, where $f$ is replaced by $\partial^k_x\partial^l_y f$. Let $m\geq 5$, $\psi(x) \equiv x^m$ (note that this function is not an admissible weight function). Multiplying this equation by $2v(t,x,y)\psi(x)$ and integrating over $\Sigma_+$, we derive an equality (here we use that $\psi^{j}(0)=0$ for $j\leq 4$)
\begin{multline}\label{2.12}
\frac{d}{dt}\iint v^2 \psi\, dxdy + 5\iint v_{xx}^2\psi'\,dxdy + \iint (3v_x^2 +v_y^2)\psi' \,dxdy \\=
5\iint v_x^2 \psi'''\, dxdy + \iint (-\psi^{(5)} +\psi''' +b\psi') v^2\, dxdy +2\iint\partial^k_x\partial^l_y f v\psi\,dxdy.
\end{multline}
Here
$$
\iint v_x^2 \psi'''\, dxdy = -\iint v_{xx} v \psi'''\,dxdy - \iint v_x v \psi^{(4)}\, dxdy.
$$
It is easy to see that $\psi''' \leq \sqrt{6\psi' \psi^{(5)}}$, $\psi^{(4)} \leq \sqrt{2\psi''' \psi^{(5)}}$ and derive from the above equality that
$$
\iint v_x^2 \psi'''\,dxdy \leq \iint v_{xx}^2 \psi' \,dxdy + 8\iint v^2 \psi^{(5)}\, dxdy.
$$
Then equality \eqref{2.12} yields that
\begin{equation}\label{2.13}
\frac{d}{dt}\iint v^2 \psi\, dxdy \leq \iint v^2(39\psi^{(5)} + \psi''' +|b|\psi')\, dxdy + 2\iint\partial^k_x\partial^l_y f v\psi\,dxdy.
\end{equation}

Fix $\alpha>0$ and let $n\geq5$. For any $m\in [5,n]$ multiplying the corresponding inequality \eqref{2.13} by $(2\alpha)^m/(m!)$ and summing by $m$ we obtain that for
$$
z_n(t)\equiv \iint\sum_{m=0}^n\frac{(2\alpha x)^m}{m!}v^2(t,x,y)\,dxdy,
$$
due to the special choice of the function $\psi$, inequalities
$$
z_n'(t)\leq c z_n(t)+c, \quad z_n(0)\leq c,
$$
hold uniformly with respect to $n$,
whence it follows that
$$
\sup_{t\in[0,T]}\iint e^{2\alpha x} v^2\,dxdy<\infty.
$$
Thus, $u\in C^\infty([0,T], \widetilde{\EuScript S}_{exp}(\overline{\Sigma}_+))$. From now on in the proof we use the notation $w(t,x,y)$ for the constructed solution of the initial value problem.

Next, let $\mu_0(t,y) \equiv -w(t,0,y)$, $\mu_1(t,y)\equiv -w_x(t,0,y)$. Note that the functions $\mu_j\in C^\infty(\overline{B}_T)$ and satisfy boundary conditions \eqref{1.4}. Moreover, the compatibility conditions from the hypothesis of the lemma ensure that $\partial_t^l \mu_j(0,y) \equiv 0$ $\forall l$. Consider in $\Pi_T^+$ an initial-boundary value problem
\begin{gather}\label{2.14}
u_t -u_{xxxxx} +u_{xxx} +u_{xyy} +bu_x =0,\\
\label{2.15}
u\big|_{t=0}=0,\quad u\big|_{x=0} = \mu_0(t,y), \quad u_x\big|_{x=0} =\mu_1(t,y),
\end{gather}
and with boundary conditions \eqref{1.4}.

Let $\Psi(t,x,y) \equiv \mu_0(t,y)\eta(1-x) + \mu_1(t,y)x\eta(1-x)$, $F(t,x,y) \equiv \Psi_{xxxxx} - \Psi_{xxx} -\Psi_{xyy} -b\Psi_x -\Psi_t$, $U(t,x,y) \equiv u(t,x,y) - \Psi(t,x,y)$, then problem \eqref{2.14}, \eqref{2.15}, \eqref{1.4} is equivalent to problem \eqref{2.9}, \eqref{1.2}--\eqref{1.4} for the function $U$, $u_0\equiv 0$, $f\equiv F$. It is obvious that $F\in C^\infty([0,T], \widetilde{\EuScript S}_{exp}(\overline{\Sigma}_+))$ and $\partial_t^l F(0,x,y) \equiv 0$ $\forall l$.

Apply the Galerkin method. Let $\{\varphi_j(x): j=1,2,\dots\}$ be a set of linearly independent functions complete in the space $\{\varphi \in H^5(\mathbb R_+): \varphi(0)=\varphi'(0)=0\}$.
Seek an approximate solution of the last problem in the form $U_k(t,x,y)= \sum\limits_{j,l=1}^k c_{kjl}(t) \varphi_j(x)\psi_l(y)$ via conditions 
\begin{multline}\label{2.16}
\iint (U_{kt}-U_{kxxxxx}+U_{kxxx}+U_{kxyy}+bU_{kx})\varphi_i(x)\psi_m(y)\,dxdy\\ -
\iint F\varphi_i\psi_m\,dxdy=0, \qquad i,m=1,\dots,k, \quad t\in [0,T],
\end{multline}
$c_{kjl}(0)=0$. 
Multiplying \eqref{2.16} by $2c_{kim}(t)$ and summing with respect to $i,m$, we find that
\begin{equation}\label{2.17}
\frac d{dt} \iint U_k^2\,dxdy +\int_0^L U_{kxx}^2\big|_{x=0}\,dy = 
2\iint F U_k\,dxdy,
\end{equation}
whence follows that
\begin{equation}\label{2.18}
\|U_k\|_{L_\infty(0,T;L_{2,+})} \leq \|F\|_{L_1(0,T;L_{2,+})}.
\end{equation}
Next, putting in \eqref{2.16} $t=0$, multiplying by $c'_{kim}(0)$ and summing with respect to $i,m$, we derive that $U_{kt}\big|_{t=0}=0$. Then differentiating \eqref{2.16} with respect to $t$, multiplying by $2c'_{kim}(t)$ and summing with respect to $i,m$, we find  similarly to  \eqref{2.18} that
\begin{equation}\label{2.19}
\|U_{kt}\|_{L_\infty(0,T;L_{2,+})} \leq \|F_t\|_{L_1(0,T;L_{2,+})}.
\end{equation}
Next, since $\psi_m^{(2n)}(y)=(-\lambda_m)^n\psi_m(y)$ it follows from \eqref{2.16} that for any $n$ and $l$
similarly to \eqref{2.18}, \eqref{2.19}
\begin{equation}\label{2.20}
\|\partial_t^l \partial_y^n U_k\|_{L_\infty(0,T;L_{2,+})} \leq \|\partial_t^l \partial_y^n F\|_{L_1(0,T;L_{2,+})}.
\end{equation}
Estimate \eqref{2.20} provides existence of a weak solution $U(t,x,y)$ to the considered problem such that $\partial_t^l \partial_y^n U \in C([0,T];L_{2,+})$ for all $l$ and $n$ in the sense of the corresponding integral equality of \eqref{1.11} type for $g\equiv0$, $f\equiv F$, $u_0\equiv 0$. 
Note, that the traces of the function $U$ satisfy conditions \eqref{1.2} for $u_0\equiv 0$ and \eqref{1.4}. 

Moreover, it follows from the  corresponding equality of \eqref{1.11} type that since
\begin{equation}\label{2.21}
U_{xxxxx} = U_t +U_{xxx} +U_{xyy} +bU_x - F,
\end{equation}
$\partial_t^l \partial_y^n U_{xxxxx} \in C([0,T];H^{(-3,0)}_+)\ \forall l,n$, therefore, the application of inequality \eqref{2.7} (for $\varphi\equiv \partial_t^l\partial_y^n U$, $k=0$, $m=3$) yields that $\partial_t^l\partial_y^n U_x \in C([0,T];L_{2,+})\ \forall l,n$, then the application twice of \eqref{2.21}, \eqref{2.7} (for $k=1$, $m=2$ and $k=2$, $m=1$) yields that $\partial_t^l \partial_y^n U_{xxx} \in C([0,T];L_{2,+})\ \forall l,n$, whence again from \eqref{2.21} follows that $\partial_t^l \partial_y^n U_{xxxxx} \in C([0,T];L_{2,+})\ \forall l,n$, the function $U$ satisfies corresponding equation \eqref{2.9} a.e. in $\Pi_T^+$ and its traces satisfy \eqref{1.3}. 
For any natural $m$ differentiating corresponding equation \eqref{2.9} $5(m-1)$ times and using induction with respect to $m$, we derive that $\partial_t^l\partial_x^{5m}\partial_y^n U \in C([0,T];L_{2,+})$. 

As a result, the solution to problem \eqref{2.14}, \eqref{2.15}, \eqref{1.4} is constructed such that $\partial_t^l\partial_x^{m}\partial_y^n u \in C([0,T];L_{2,+})$ $\forall l,m,n$. From now on in the proof we use notation $v(t,x,y)$ for this solution.

Finally, the function $u(t,x,y) \equiv w(t,x,y) + v(t,x,y)$ is the solution to problem \eqref{2.9}, \eqref{1.2}--\eqref{1.4} such that $\partial_t^l\partial_x^{m}\partial_y^n u \in C([0,T];L_{2,+})$ $\forall l,m,n$. Let $\widetilde u(t,x,y) \equiv u(t,x,y)\eta(x-1)$. The function $\widetilde u$ solves an initial value problem in the strip $\Sigma$ of \eqref{2.9}, \eqref{1.2}, \eqref{1.4} type, where the functions $f$, $u_0$ are substituted by corresponding functions $\widetilde f$, $\widetilde u_0$ from the same classes and the obtained result at the beginning of the proof for the initial value problem together with the obvious uniqueness provide that $\widetilde u\in C^\infty([0,T], \widetilde{\EuScript S}_{exp}(\overline{\Sigma}_+))$ and so $u\in C^\infty([0,T], \widetilde{\EuScript S}_{exp}(\overline{\Sigma}_+))$.
\end{proof}

\begin{lemma}\label{L2.5}
A generalized solution to problem \eqref{2.9}, \eqref{1.2}--\eqref{1.4} is unique in the space $L_2(\Pi_T^+)$.
\end{lemma}

\begin{proof}
This lemma is a corollary of the following result on existence of smooth solutions to the corresponding adjoint problem.
\end{proof}

Consider in $\Pi_T^+$ an initial-boundary value problem for an equation
\begin{equation}\label{2.22}
u_t+u_{xxxxx}-u_{xxx}-u_{xyy}-bu_x=f(t,x,y)
\end{equation}
with initial condition \eqref{1.2}, boundary conditions \eqref{1.4} and boundary conditions
\begin{equation}\label{2.23}
u\big|_{x=0}= u_x\big|_{x=0}=  u_{xx}\big|_{x=0} =0.
\end{equation}

\begin{lemma}\label{L2.6}
Let $u_0\in \widetilde{\EuScript S}(\overline{\Sigma}_+)$, $f\in C^\infty\bigl([0,T]; \widetilde{\EuScript S}(\overline{\Sigma}_+)\bigr)$ and $\widetilde\Phi_j(0,y) = \widetilde \Phi_{jx}(0,y)\equiv 0$ for any $j$, where here in the definition of the corresponding functions $\widetilde\Phi_j$ in comparison with \eqref{2.10} the sign before the second term in the right-hand side is changed.
Then there exists a unique solution to problem \eqref{2.22}, \eqref{1.2}, \eqref{2.23}, \eqref{1.4} $u\in C^\infty\bigl([0,T]; \widetilde{\EuScript S}(\overline{\Sigma}_+)\bigr)$.
\end{lemma}

\begin{proof}
First of all, as in the proof of Lemma \ref{L2.4} extend the functions $u_0$ and $f$ to the whole strip and consider problem \eqref{2.22}, \eqref{1.2}, \eqref{1.4}. Construct its solution $w \in C^\infty([0,T],\widetilde{\EuScript S}(\overline{\Sigma}))$ in a similar way with the only obvious difference in \eqref{2.11}.

Next, let $\mu_0(t,y) \equiv -w(t,0,y)$, $\mu_1(t,y)\equiv -w_x(t,0,y)$, $\mu_2(t,y) \equiv - w_{xx}(t,0,y)$. Note that the functions $\mu_j\in C^\infty(\overline{B}_T)$ and satisfy boundary conditions \eqref{1.4}. Moreover, the compatibility conditions form the hypothesis of the lemma ensure that $\partial_t^l \mu_j(0,y) \equiv 0$ $\forall l$. Consider in $\Pi_T^+$ an initial-boundary value problem
\begin{gather}\label{2.24}
u_t +u_{xxxxx} -u_{xxx} -u_{xyy} -bu_x =0,\\
\label{2.25}
u\big|_{t=0}=0,\quad u\big|_{x=0} = \mu_0(t,y), \quad u_x\big|_{x=0} =\mu_1(t,y),  \quad u_{xx}\big|_{x=0} =\mu_2(t,y),
\end{gather}
and with boundary conditions \eqref{1.4}.

Let $\Psi(t,x,y) \equiv \mu_0(t,y)\eta(1-x) + \mu_1(t,y)x\eta(1-x) + \mu_2(t,y)x^2\eta(1-x)/2$, $F(t,x,y) \equiv -\Psi_{xxxxx} +\Psi_{xxx} +\Psi_{xyy} +b\Psi_x -\Psi_t$, $U(t,x,y) \equiv u(t,x,y) - \Psi(t,x,y)$, then problem \eqref{2.24}, \eqref{2.25}, \eqref{1.4} is equivalent to problem \eqref{2.22}, \eqref{1.2}, \eqref{2.23}, \eqref{1.4} for the function $U$, $u_0\equiv 0$, $f\equiv F$. It is obvious that $F\in C^\infty([0,T], \widetilde{\EuScript S}(\overline{\Sigma}_+))$ and $\partial_t^l F(0,x,y) \equiv 0$ $\forall l$.

Let $\{\varphi_j(x): j=1,2,\dots\}$ be the same set of functions as in the proof of Lemma \ref{L2.4}. Seek an approximate solution in the form $U_k(t,x,y)= \sum\limits_{j,l=1}^k c_{kjl}(t) \varphi_j(x)\psi_l(y)$ via conditions 
\begin{multline}\label{2.26}
\iint \bigl(U_{kt}\varphi_i\psi_m -U_k(\varphi_i^{(5)}\psi_m-\varphi'''_i\psi_m -\varphi'_i\psi''_m-b\varphi_i'\psi_m)\bigr)\,dxdy \\-
\iint F\varphi_i\psi_m\,dxdy=0, \quad i,m=1,\dots,k, t\in [0,T],
\end{multline}
$c_{kjl}(0)=0$. Multiplying \eqref{2.26} by $2c_{kim}(t)$ and summing with respect to $i,m$, we derive equality \eqref{2.17}, which implies estimate \eqref{2.18} Then in similar way we derive estimate \eqref{2.20}, which provide existence of a weak solution $U(t,x,y)$ to the considered problem such that $\partial_t^l \partial_y^n U \in C([0,T];L_{2,+})\ \forall l,n\geq 0$ in the following sense: for any function $\phi\in L_\infty(0,T;\widetilde H_+^2)$, such that $\phi_t, \phi_{xxxxx}, \phi_{xyy}\in L_\infty(0,T;L_{2,+})$, $\phi\big|_{t=T}=0$, $\phi\big|_{x=0}= \phi_x\big|_{x=0}=0$, the following equality holds:
\begin{equation}\label{2.27}
\iiint_{\Pi_T^+}\Bigl[U(\phi_t +\phi_{xxxxx}-\phi_{xxx}-\phi_{xyy}-b\phi_x) +F\phi\Bigr]\,dxdydt =0.
\end{equation}
Then also similarly to the proof of Lemma \ref{L2.4} we obtain a solution to problem \eqref{2.24}, \eqref{2.25}, \eqref{1.4} $v$ such that $\partial_t^l\partial_x^{m}\partial_y^n v \in C([0,T];L_{2,+})$ $\forall l,m,n$.

Finally, again similarly to the proof of Lemma \ref{L2.4} we show that the function $u\equiv w+v$ is the desired solution.
\end{proof}

\begin{remark}\label{R2.1}
In further lemmas of this section we first consider smooth solutions constructed in Lemma~\ref{L2.4} and then pass to the limit on the basis of obtained estimates. 
\end{remark}

\begin{lemma}\label{L2.7}
Let $\psi(x)$ be an admissible weight function, such that $\psi'(x)$ is also an admissible weight function, $u_0\in L_{2,+}^{\psi(x)}$, $f\equiv f_0 +f_{1x}$, where   $f_0\in L_1(0,T;L_{2,+}^{\psi(x)})$, $f_1\in L_{4/3}(0,T;L_{2,+}^{\psi^{3/2}(x)(\psi'(x))^{-1/2}})$. Then there exist a (unique) weak solution to problem \eqref{2.9}, \eqref{1.2}--\eqref{1.4} from the space $X^{\psi(x)}(\Pi_T^+)$ and a function $\mu_2\in L_2(B_T)$, such that for any function $\phi\in L_\infty(0,T;\widetilde H_+^2)$, $\phi_t, \phi_{xxxxx}, \phi_{xyy}\in L_\infty(0,T:L_{2,+})$, $\phi\big|_{t=T}=0$, $\phi\big|_{x=0} =\phi_x\big|_{x=0}=0$, the following equality holds:
\begin{multline}\label{2.28}
\iiint_{\Pi_T^+}\Bigl[u(\phi_t -\phi_{xxxxx} +\phi_{xxx}+\phi_{xyy} +b\phi_x)+(f_0+f_2)\phi -f_1\phi_x\Bigr]\,dxdydt \\
+\iint u_0\phi\big|_{t=0}\,dxdy -
\iint_{B_T} \mu_2\phi_{xx}\big|_{x=0}\,dydt =0.
\end{multline}
Moreover, for $t\in (0,T]$
\begin{equation}\label{2.29}
\|u\|_{X^{\psi(x)}(\Pi_t^+)} +\|\mu_2\|_{L_2(B_t)}\leq c(T)\Bigl(\|u_0\|_{L_{2,+}^{\psi(x)}} +\|f_0\|_{L_1(0,t;L_{2,+}^{\psi(x)})}  +\|f_1\|_{L_{4/3}(0,t;L_{2,+}^{\psi^{3/2}(x)(\psi'(x))^{-1/2}})} \Bigr),
\end{equation}
and for a.e. $t\in (0,T)$
\begin{multline}\label{2.30}
\frac{d}{dt}\iint u^2(t,x,y)\psi(x)\,dxdy  + \iint (5u_{xx}^2 +3u_x^2 +u_y^2 -bu^2)\psi'\,dxdy  -
\iint (5u_x^2 +u^2)\psi'''\, dxdy \\+\iint u^2 \psi^{(5)}\, dxdy + \psi(0)\int_0^L \mu_2^2\,dy = 
2\iint f_0 u\psi\,dxdy 
-2\iint f_1 (u\psi)_x\,dxdy.
\end{multline}
If $f_1\equiv 0$, then in equality \eqref{2.30} one can put $\psi\equiv 1$.
\end{lemma}

\begin{proof}
Multiplying \eqref{2.9} by $2u(t,x,y)\rho(x)$, where either $\rho(x)\equiv \psi(x)$ or $\rho(x)\equiv 1$, and integrating over $\Sigma_+$ we derive equality \eqref{2.30} for $\mu_2 \equiv u_{xx}\big|_{x=0}$, where $\psi$ is substituted by $\rho$.
Here according to \eqref{2.2} (and since $\psi' \leq c\psi$) for arbitrary $\varepsilon >0$
\begin{multline}\label{2.31}
\Bigl|\iint f_1(u\psi)_x\,dxdy\Bigr| \leq c \|(|u_{x}|+|u|)(\psi')^{1/4} \psi^{1/4}\|_{L_{2,+}}\|f_1 \psi^{3/4} (\psi')^{-1/4}\|_{L_{2,+}} \\ \leq 
c_1 \Bigl[ \bigl\|\bigl(|u_{xx}|+|u_y|\bigr)(\psi')^{1/2}\|_{L_{2,+}}^{1/2}\|u\psi^{1/2}\|_{L_{2,+}}^{1/2} +
\|u\psi^{1/2}\|_{L_{2,+}}\Bigr] \|f_1 \psi^{3/4} (\psi')^{-1/4}\|_{L_{2,+}} \\ \leq
\varepsilon \iint \bigl(u_{xx}^2+u_y^2 \bigr)\psi'\,dxdy +
c(\varepsilon)\|f_1\|_{L_{2,+}^{\psi^{3/2}(x)(\psi'(x))^{-1/2}}}^{4/3}\Bigl(\iint u^2\psi\,dxdy\Bigr)^{1/3}  \\ +
c_1 \|f_1\|_{L_{2,+}^{\psi^{3/2}(x)(\psi'(x))^{-1/2}}} \Bigl(\iint u^2\psi\,dxdy\Bigr)^{1/2},
\end{multline}
and according to \eqref{2.5} (where $\psi$ is substituted by $\psi'$)
\begin{equation}\label{2.32}
\Bigl|\iint u_x^2 \psi'''\,dxdy\Bigr| \leq \varepsilon \iint u_{xx}^2\psi'\,dxdy  +
c(\varepsilon) \iint u^2\psi\,dxdy.
\end{equation}
Equality \eqref{2.30} and inequalities \eqref{2.31}, \eqref{2.32} imply that for smooth solutions
\begin{equation}\label{2.33}
\|u\|_{X^{\psi(x)}(\Pi_T^+)} + \|u_{xx}\big|_{x=0}\|_{L_2(B_T)} \leq c.
\end{equation}
The end of the proof is standard.
\end{proof}

\begin{remark}\label{R2.2}
The method of construction of weak solution in Lemma~\ref{L2.7} via closure ensures that $u|_{x=0}= u_x|_{x=0} =0$ in the trace sense (this fact can be also easily derived from equality \eqref{2.28}, since $u_{xx} \in L_2((0,T)\times (0,1)\times (0,L))$). Moreover, if it is known, in addition, that $u_{xx} \in C_w([0,x_0];L_2(B_T))$ for certain $x_0>0$, then equality \eqref{2.28} yields that $u_{xx}|_{x=0}=\mu_2$ (for example, one can put $\phi \equiv x\eta(1-x/h)\omega(t,y)$ for $h>0$ and any $\omega\in C_0^\infty(B_T)$ and then tend $h$ to zero).
\end{remark}

\begin{lemma}\label{L2.8}
Let $\psi(x)$ be an admissible weight function such that $\psi'(x)$ is also admissible, $u_0\in \widetilde H_{+}^{1,\psi(x)}$, $u_0(0,y)\equiv 0$, $f\equiv f_0+f_1$, where $f_0\in L_2(0,T; \widetilde H_+^{1,\psi(x)})$, $f_1\in L_2(0,T;L_{2,+}^{\psi^2(x)/\psi'(x)})$ for certain $T>0$. Then there exist a strong solution $u\in X^{1,\psi(x)}(\Pi_T^+)$ to problem \eqref{2.9}, \eqref{1.2}--\eqref{1.4} and a function $\mu_4\in L_2(B_T)$, such that for any $t\in (0,T)$
\begin{equation}\label{2.34}
\|u\|_{X^{1,\psi(x)}(\Pi_t^+)} +\|\mu_4\|_{L_2(B_t)} 
\leq c(T) \Bigl(\|u_0\|_{\widetilde H^{1,\psi(x)}_+}
+\|f_0\|_{L_2(0,t;\widetilde H^{1,\psi(x)}_+)}
+\|f_1\|_{L_2(0,t;L_{2,+}^{\psi^2(x)/\psi'(x)})}\Bigr),
\end{equation}
and for a.e. $t\in (0,T)$
\begin{multline}\label{2.35}
\frac{d}{dt}\iint(u_{xx}^2+u_x^2+u_y^2)\psi \,dxdy 
+\int_0^L (\mu_4^2\psi +4\mu_4u_{xxx}\psi' +2\mu_4u_{xx}\psi'' -2\mu_4u_{xx}\psi -3u_{xxx}^2\psi'' \\-2u_{xxx}u_{xx}\psi'''  +4u_{xxx}u_{xx}\psi'+u_{xx}^2\psi^{(4)}+(b+1)u_{xx}^2\psi  - 4u_{xx}^2\psi'')\big|_{x=0}\,dy \\
+\iint(5u_{xxxx}^2+6u^2_{xxy}+u^2_{yy} +8u_{xxx}^2 +4u_{xy}^2 + 3(b+1)u_{xx}^2 -bu_x^2 -bu_y^2)\psi' \,dxdy \\ 
 = \iint(5u^2_{xxx} +5u_{xy}^2 +6u_{xx}^2 +u_y^2 +(b+1)u_x^2)\psi'''\,dxdy \\ -
 \iint (u_{xx}^2+u_x^2 +u_y^2) \psi^{(5)}\, dxdy 
+2\iint(f_{0xx}u_{xx}+f_{0x}u_x+f_{0y}u_y)\psi \,dxdy  \\-
2\int_0^L  [(f_0(u_{xx}\psi)_x -f_{0x}u_{xx}\psi]\big|_{x=0}\,dy 
+2\iint f_1[(u_{xx}\psi)_{xx}-(u_x\psi)_x-u_{yy}\psi]\,dxdy.
\end{multline}
If $f_1\equiv 0$, then in equality \eqref{2.35} one can put $\psi(x)\equiv 1$.
\end{lemma}

\begin{proof}
Multiplying \eqref{2.9} by $2\bigl(u_{xx}\rho(x)\bigr)_{xx} -2\bigl(u_x\rho(x)\bigr)_x -2u_{yy}\rho(x)$, where either $\rho(x)\equiv \psi(x)$ or $\rho(x)\equiv 1$, and integrating over $\Sigma_+$ we derive equality \eqref{2.35} for $\mu_4 \equiv u_{xxxx}\big|_{x=0}$, where $\psi$ is substituted by $\rho$.
Here according to \eqref{2.6} for an arbitrary $\varepsilon>0$
\begin{equation}\label{2.36}
\int_0^L u_{xxx}^2\big|_{x=0}\,dy \leq \varepsilon\iint u_{xxxx}^2\psi'\,dxdy + c(\varepsilon)\iint u_{xx}^2\psi\,dxdy,
\end{equation}
similarly to \eqref{2.32}
\begin{equation}\label{2.37}
\Bigl|\iint (u_{xxx}^2 + u_{xy}^2)\psi'''\, dxdy\Bigr| \leq \varepsilon\iint (u_{xxxx}^2 +u_{xxy}^2)\psi' \,dxdy + c(\varepsilon)\iint (u_{xx}^2 + u_y^2)\psi\, dxdy,
\end{equation}
and
\begin{equation}\label{2.38}
\Bigl|\iint f_1[(u_{xx}\psi)_{xx} -u_{yy}\psi]\,dxdy\Bigr|  \leq
\varepsilon \iint \bigl(u_{xxxx}^2+u_{yy}^2 +u_{xx}^2\bigr)\psi'\,dxdy +
c(\varepsilon)\iint f_1^2 \psi^2 (\psi')^{-1} \,dxdy.
\end{equation}
Equality \eqref{2.35} and inequalities \eqref{2.36}--\eqref{2.38} together with \eqref{2.33} imply that for smooth solutions
$$
\|u\|_{X^{1,\psi(x)}(\Pi_T^+)} + \|u_{xxxx}\big|_{x=0}\|_{L_2(B_T)} \leq c.
$$
\end{proof}

\begin{lemma}\label{L2.9}
Let the hypothesis of Lemma~\ref{L2.8} be satisfied for $\psi(x) \equiv e^{2\alpha x}$ for certain $\alpha>0$. Let $g\in C^1(\mathbb R)$, $g(0)=0$. Consider the strong solution $u\in X^{1,\psi(x)}(\Pi_T^+)$ to problem \eqref{2.9}, \eqref{1.2}--\eqref{1.4}. Then for a.e. $t\in (0,T)$
\begin{multline}\label{2.39}
\frac{d}{dt}\iint g^*(u)\rho\,dxdy +\iint g'(u)u_x(u_{xxxx} -u_{xx} -u_{yy})\rho\,dxdy  +\iint g(u)(u_{xxxx} -u_{xx} -u_{yy})\rho'\,dxdy \\
-b\iint g^*(u)\rho'\,dxdy 
= \iint g(u)f\rho \,dxdy,
\end{multline}
where either $\rho(x)\equiv 1$ or $\rho(x)$ is an admissible weight function such that $\rho(x) \leq c\psi(x)$ $\forall x\geq 0$.
\end{lemma}

\begin{proof}
In the smooth case equality \eqref{2.39} is obtained via multiplication of \eqref{2.9} by $g(u(t,x,y))\psi(x)$ and subsequent integration and in the general case via closure, which here is easily justified since  $X^{1,\psi(x)}(\Pi_T^+) \subset L_\infty(\Pi_T^+)$ and $\psi' \sim \psi$. 
\end{proof}

\section{Existence of solutions}\label{S3}

First of all establish an auxiliary result on existence of weak solutions for a regularized equation.

\begin{lemma}\label{L3.1}
Let $g\in C^1(\mathbb R)$, $g(0)=0$, $|g'(u)|\leq c\ \forall u\in\mathbb R$, $\psi(x) \equiv e^{2\alpha x}$ for certain $\alpha>0$, $u_0\in L_{2,+}^{\psi(x)}$, $f\in L_1(0,T;L_{2,+}^{\psi(x)})$. Then problem \eqref{1.1}--\eqref{1.4} has a unique weak solution $u\in X^{\psi(x)}(\Pi_T^+)$.
\end{lemma}

\begin{proof}
We apply the contraction principle. For $t_0\in(0,T]$ define a mapping $\Lambda$ on $X^{\psi(x)}(\Pi_{t_0}^+)$ as follows: $u=\Lambda v\in X^{\psi(x)}(\Pi_{t_0}^+)$ is a weak solution to a linear problem
\begin{equation}\label{3.1}
u_t -u_{xxxxx}+u_{xxx}+u_{xyy} +bu_x =f-\bigl(g(v)\bigr)_x
\end{equation}
in $\Pi_{t_0}^+$ with initial and boundary conditions \eqref{1.2}--\eqref{1.4}.

Note that $\psi^{3/2}(\psi')^{-1/2} \sim \psi$, $|g(v)|\leq c|v|$ and, therefore,
Lemma~\ref{L2.7} (where $f_0\equiv f$, $f_1 \equiv g(v)$) provides that the mapping $\Lambda$ exists. Moreover, for functions 
$v,\widetilde{v}\in X^{\psi(x)}(\Pi_{t_0}^+)$
according to inequality \eqref{2.29} 
\begin{gather*}
\|\Lambda v\|_{X^{\psi(x)}(\Pi_{t_0}^+)} \leq c(T)
\Bigl(\|u_0\|_{L_{2,+}^{\psi(x)}} +\|f\|_{L_1(0,T;L_{2,+}^{\psi(x)})} +
t_0^{3/4}\|v\|_{X^{\psi(x)}(\Pi_{t_0}^+)}\Bigr),\\
\|\Lambda v-\Lambda\widetilde{v}\|_{X^{\psi(x)}(\Pi_{t_0}^+)}\leq
c(T) t_0^{3/4}\|v-\widetilde{v}\|_{X^{\psi(x)}(\Pi_{t_0}^+)},
\end{gather*}
whence first the local result succeeds. Next, since the constant in the right-hand side in the above inequalities is uniform with respect to $u_0$ and$f$, one can extend the solution to the whole time segment $[0,T]$ by the standard argument.
\end{proof}

Now we pass to the results of existence in Theorem~\ref{T1.1}.

\begin{proof}[Proof of Existence Part of Theorem~\ref{T1.1}]
For $h\in (0,1]$ consider a set of initial-boundary value problems
\begin{equation}\label{3.2}
u_t -u_{xxxxx} +u_{xxx}+u_{xyy} +bu_x+g'_h(u)u_x = f_h,
\end{equation}
with an initial condition
\begin{equation}\label{3.3}
u\big|_{t=0} = u_{0h}(x)
\end{equation}
and boundary conditions \eqref{1.3}, \eqref{1.4}, where 
\begin{equation}\label{3.4}
f_h(t,x,y) \equiv f(t,x,y)\eta(1/h-x), \quad u_{0h}(x,y)\equiv u_0(x)\eta(1/h-x),
\end{equation}
\begin{equation}\label{3.5}
g'_h(u)\equiv g'(u) \eta (2-h|u|),\quad g_h(u) \equiv \int_0^u g'_h(\theta)\, d\theta.
\end{equation}
Note that $g_h(u)= g(u)$ if $|u|\leq 1/h$, $g'_h(u) =0$ if $|u|\geq 2/h$, $|g'_h(u)| \leq c(h)$ $\forall u$ and the functions $g_h$ satisfy inequality \eqref{1.12} uniformly with respect to $h$.

Lemma~\ref{L3.1} implies that there exists a unique solution to this problem $u_h\in X^{e^{2\alpha x}}(\Pi_T^+)$ for any $\alpha>0$.

Next, establish appropriate estimates for functions $u_h$ uniform with respect to~$h$ (we drop the subscript $h$ in intermediate steps for simplicity). First, note that $g'(u)u_x \in L_1(0,T;L_{2,+}^{\psi(x)})$ and so the hypothesis of Lemma~\ref{L2.7} is satisfied (for $f_1=f_2\equiv 0$). Then equality \eqref{2.30} provides that for both for $\rho(x)\equiv 1$ and $\rho(x) \equiv \psi(x)$
\begin{multline}\label{3.6}
\frac{d}{dt}\iint u^2\rho\,dxdy  + \iint (5u_{xx}^2 +3u_x^2 +u_y^2 -bu^2)\rho'\,dxdy   -
\iint (5u_x^2 +u^2)\rho'''\, dxdy +\iint u^2 \rho^{(5)}\, dxdy \\+ \rho(0)\int_0^L \mu_2^2\,dy = 
2\iint fu\rho\,dxdy + 2\iint \bigl(g'(u)u\bigr)^* \rho'\, dxdy.
\end{multline}
Choosing $\rho\equiv 1$, we obtain, that uniformly with respect to $h$ (and also uniformly with respect to $L$)
\begin{equation}\label{3.7}
\|u_h\|_{C([0,T];L_{2,+})} \leq c.
\end{equation}
Now choose $\rho(x) \equiv \psi(x)$. Note that uniformly with respect to $h$
\begin{equation}\label{3.8}
\bigl|\bigl(g_h'(u)u\bigr)^*\bigr| \leq c|u|^{p+2}.
\end{equation}
Let$q=p+2$, $s=s(0,q)$ from \eqref{2.1}, $\psi_1(x) \equiv \psi'(x)$, $\psi_2(x) \equiv \bigl(\psi'(x)\bigr)^{\frac{2(1-qs)}{q(1-2s)}}$ (note that $qs=3p/8<1$). 
Applying interpolating inequality \eqref{2.2}, we obtain that (and this is the central item of the proof)
\begin{multline}\label{3.9}
\iint |u|^{p+2}\psi'\,dxdy = \iint |u|^{q} \psi_1^{qs} \psi_2^{q(\frac12-s)} \,dxdy \\ \leq 
c\Bigl(\iint (u_{xx}^2 + u_y^2 +u^2)\psi_1 \,dxdy\Bigr)^{qs} \Bigl(\iint u^2 \psi_2 \,dxdy\Bigr)^{q(\frac12-s)} \\ =
c\Bigl(\iint (u_{xx}^2 + u_y^2 +u^2)\psi' \,dxdy\Bigr)^{qs} \Bigl(\iint (u^2\psi')^{\frac{2(1-qs)}{q(1-2s)}} u^{\frac{2(q-2)}{q(1-2s)}}\, dxdy \Bigr)^{q(\frac12-s)} \\ \leq 
c\Bigl(\iint (u_{xx}^2 + u_y^2 +u^2)\psi' \,dxdy\Bigr)^{3p/8} \Bigl(\iint u^2\psi' \,dxdy\Bigr)^{(8-3p)/8} \Bigl(\iint u^2 \,dxdy\Bigr)^{p/2}
\end{multline}
(note that here the constant $c$ is also uniform with respect to $L$ in the cases a) and c)).
Since the norm of the functions $u_h$ in the space $L_{2,+}$ is already  estimated in \eqref{3.7}, it follows from \eqref{3.6}, \eqref{3.8}, \eqref{3.9} (also with use of \eqref{2.32}) that uniformly with respect to $h$
\begin{equation}\label{3.10}
\|u_h\|_{X^{\psi(x)}(\Pi_T^+)} \leq c.
\end{equation}

Write the analogue of \eqref{3.6}, where $\rho(x)$ is substituted by $\rho_0(x-x_0)$ for any $x_0\geq 0$. Then it easily follows that (see \eqref{1.10})
\begin{equation}\label{3.11}
\lambda^+ (u_{hxx};T) + \lambda^+ (u_{hy};T)\leq c. 
\end{equation}
Let $\Sigma_n = (0,n)\times (0,L)$. It follows from \eqref{3.11} and \eqref{2.3} that uniformly with respect to $h$
\begin{equation}\label{3.12}
\|u_h\|_{L_{8/3}(0,T;L_\infty(\Sigma_n))} \leq c(n),
\end{equation}
and, consequently,
\begin{equation}\label{3.13}
\|g_h(u_h)\|_{L_{8/(3p)}(0,T;L_2(\Sigma_n))} \leq c(n).
\end{equation}
Then from equation \eqref{1.1} itself it follows that uniformly with respect to $h$
\begin{equation}\label{3.14}
\|u_{ht}\|_{L_1(0,T;H^{-5}(\Sigma_n))}\leq c.
\end{equation}
Since the embedding $H^1(\Sigma_n) \subset L_2(\Sigma_n)$ is compact, it follows from \cite[Section 9, Corollary 6]{Si} that the set $u_h$ is relatively compact in $L_q(0,T;L_2(\Sigma_n))$ for $q<+\infty$.

Extract a subsequence of the functions $u_h$, again denoted as $u_h$, such that as $h\to +0$
\begin{align*}
u_h\rightharpoonup u&\quad  *-\text{weakly in}\quad
L_\infty(0,T; L_{2,+}^{\psi(x)}),\\
u_{hxx}, u_{hy}\rightharpoonup u_{xx}, u_y&\quad  \text{weakly in}\quad
L_2(0,T;L_{2,+}^{\psi'(x)}),\\
u_h\rightarrow u &\quad \text{strongly in}\quad
L_{\max(2,8/(8-3p))}(0,T;L_2(\Sigma_n)) \quad \forall n.
\end{align*}
Let $\phi$ is a test function from Definition \ref{D1.1} with $\supp\phi \subset \overline{\Sigma}_n$. Then since
$$
|g_h(u_h) -g_h(u)| \leq c\bigl(|u_h|^p+|u|^p\bigr)|u_h-u|
$$
with the use of \eqref{3.12} we obtain that the limit function $u$ verifies \eqref{1.11}.

Finally, note that $g(u)\phi_x \in L_\infty(0,T; L_{1,+})$ if $p\leq 1$, while in the case $p>1$
\begin{multline*}
\|g(u)\phi\|_{L_1(\Pi_T^+)} \leq c\int_0^T \|u(\psi')^{3/8}\psi^{1/8}\|_{L_{\infty,+}}^p \iint |u\phi_x| (\psi')^{-3p/8}\psi^{-p/8}\,dxdy\,dt \\ \leq 
c_1 \int_0^T \Bigl(\iint (u_{xx}^2+u_y^2 +u^2)\psi'\, dxdy\Bigr)^{3p/8} \Bigl(\iint u^2\psi\, dxdy\Bigr)^{(p+4)/8} \\ \times
\Bigl(\iint \phi_x^2 (\psi')^{-3p/4}\psi^{-1-p/4}\, dxdy\Bigr)^{1/2} dt<+\infty,
\end{multline*}
since $(\psi')^{-3p/4}\psi^{-1-p/4} \leq c(1+x)^{3pn/4}$ by virtue of the additional property of the function $\psi$. Approximating any test function from Definition \ref{D1.1} by the compactly supported ones and passing to the limit we obtain equality \eqref{1.1} in the general case.
\end{proof}

Now we pass to strong solutions and first establish a local result on existence for the regularized equation.

\begin{lemma}\label{L3.2}
Let $g\in C^2(\mathbb R)$, $g(0)=0$, $|g'(u)|, |g''(u)| \leq c\ \forall u\in\mathbb R$, $\psi(x) \equiv e^{2\alpha x}$ for certain $\alpha>0$, $u_0\in \widetilde H_+^{1,\psi(x)}$, $u_0(0,y)\equiv 0$, $f\in L_2(0,T;\widetilde H_+^{1,\psi(x)})$. Then there exists $t_0\in (0,T]$, such that problem \eqref{1.1}--\eqref{1.4} has a unique strong solution $u\in X^{1,\psi(x)}(\Pi_{t_0}^+)$.
\end{lemma}

\begin{proof}
Similarly to the proof of Lemma \ref{L3.1} we construct the desired solution as a fixed point of the map $\Lambda$ but defined on the space $X^{1,\psi(x)}(\Pi_{t_0}^+)$. Here $\psi^2/\psi' \sim \psi$ and Lemma \ref{L2.8}, where $f_0\equiv f$, $f_1\equiv g'(v)v_x$, ensures that such a map exists. Moreover, for functions 
$v,\widetilde{v}\in X^{1,\psi(x)}(\Pi_{t_0}^+)$ according to inequality \eqref{2.35} 
$$
\|\Lambda v\|_{X^{1,\psi(x)}(\Pi_{t_0}^+)} \leq c(T)
\Bigl(\|u_0\|_{\widetilde H_+^{1,\psi(x)}} +\|f\|_{L_2(0,T;\widetilde H_+^{1,\psi(x)})} +
t_0^{1/2}\|v\|_{X^{1,\psi(x)}(\Pi_{t_0}^+)}\Bigr),
$$
and, since $|g'(v)v_x - g'(\widetilde v)\widetilde v_x| \leq c \bigl(|v_x| +|\widetilde v_x|\bigr)|v-\widetilde v| +c|v_x - \widetilde v_x|$,
$$
\|\Lambda v-\Lambda\widetilde{v}\|_{X^{1,\psi(x)}(\Pi_{t_0}^+)} \leq
c(T) t_0^{1/2}\bigl(\|v\|_{X^{1,\psi(x)}(\Pi_{t_0}^+)} + \|\widetilde v\|_{X^{1,\psi(x)}(\Pi_{t_0}^+)}\bigr) \|v-\widetilde{v}\|_{X^{1,\psi(x)}(\Pi_{t_0}^+)},
$$
whence the assertion of the lemma succeeds.
\end{proof}

\begin{proof}[Proof of Existence Part of Theorem~\ref{T1.2}]
First establish the following a priori estimate: if $u\in X_w^{1,e^{2\alpha x}}(\Pi_{T'}^+)$, $\alpha>0$,  is a solution to problem \eqref{1.1}--\eqref{1.4} for some $T'\in (0,T]$, where the function $g\in C^2(\mathbb R)$ verifies \eqref{1.12}, then for any admissible function $\psi(x)$, such that $\psi'$ is also admissible and $\psi(x)\leq c e^{2\alpha x}$ $\forall x\geq 0$,
\begin{equation}\label{3.15}
\|u\|_{X_w^{1,\psi(x)}(\Pi_{T'}^+)} \\ \leq c\bigl(T,\|u_0\|_{\widetilde H_+^{1,\psi(x)}}, \|f\|_{L_2(0,T;\widetilde H_+^{1,\psi(x)})}\bigr).
\end{equation}

Apply equality \eqref{3.6} (here, of course, $\mu_2 = u_{xx}|_{x=0}$), then similarly to \eqref{3.7}, \eqref{3.10} 
\begin{equation}\label{3.16}
\|u\|_{X_w^{\psi(x)}(\Pi_{T'}^+)}  +\|u_{xx}\big|_{x=0}\|_{L_2(B_{T'})} \leq c.
\end{equation}

Next, since the hypotheses of Lemma~\ref{L2.8} and, consequently, Lemma~\ref{L2.9} are satisfied, write the corresponding analogues of equalities \eqref{2.35}, \eqref{2.39},   and subtract from the first one the doubled second one, then with the use of \eqref{2.36}, \eqref{2.37} we derive that
\begin{multline}\label{3.17}
\frac{d}{dt}\iint(u_{xx}^2+u_x^2+u_y^2 -2g^*(u))\rho \,dxdy  
+\iint(4u_{xxxx}^2+5u^2_{xxy}+u^2_{yy} +8u_{xxx}^2 +4u_{xy}^2)\rho' \,dxdy \\
 \leq c\iint (u_{xx}^2+u_x^2 +u_y^2)\rho\, dxdy 
 +c\iint(f_{xx}^2+f_x^2 +f_{y}^2 +f^2) \rho \,dxdy 
 -2\iint g(u)f\rho\, dxdy\\
+\iint g(u)\bigl[(6u_{xxxx} -4u_{xx} -2u_{yy})\rho' +  (8u_{xxx}\rho'' +2u_{xx}\rho''' -2u_x \rho'')\bigr]\,dxdy \\
-2b\iint g^*(u)\rho'\, dxdy - \iint \bigl(g'(u)g(u)\bigr)^*\rho' \,dxdy
+c\int_0^L u_{xx}^2\big|_{x=0} \,dy.
\end{multline}
Choose $\rho(x)\equiv 1$. Note that condition \eqref{1.12} together with \eqref{3.16}  imply that 
\begin{equation}\label{3.18}
\iint |g^*(u)|\, dxdy \leq c\|u\|^p_{L_{\infty,+}} \|u\|^2_{L_{2,+}} \leq c_1\Bigl(\iint (u_{xx}^2 +u_y^2+u^2)\,dxdy\Bigr)^{3p/8},
\end{equation}
\begin{equation}\label{3.19}
\iint g(u)f\, dxdy \leq c\|u\|^p_{L_{\infty,+}}  \|u\|_{L_{2,+}}  \|f\|_{L_{2,+}}.
\end{equation}
Therefore, equality \eqref{3.17} in this case ensures that
\begin{equation}\label{3.20}
\|u_{xx}\|_{L_\infty(0,T';L_{2,+})} + \|u_{y}\|_{L_\infty(0,T';L_{2,+})} \leq c.
\end{equation}
In particular,
\begin{equation}\label{3.21}
\|u\|_{L_\infty(\Pi_{T'}^+)} \leq c.
\end{equation}
Return to equality \eqref{3.17} and choose $\rho(x) \equiv \psi(x)$. By virtue of \eqref{3.21} $|g(u)| \leq c|u|$ and then estimate \eqref{3.15} easily follows.

Note also that from \eqref{3.17}, where $\rho(x)\equiv \rho_0(x-x_0)$ for any $x_0\geq 0$, easily follows that (see \eqref{1.10})
\begin{equation}\label{3.22}
\lambda^+ (u_{xxxx};T')  +\lambda^+(u_{xxy};T') + \lambda^+ (u_{yy};T') \leq c.
\end{equation}

To finish the proof consider set of initial-boundary value problems \eqref{3.2}, \eqref{3.3}, \eqref{1.3}, \eqref{1.4}. Lemma~\ref{L3.2} imply that for any $h\in (0,1]$ there exists a solution to such a problem $u_h \in X^{1,\psi(x)}(\Pi_{t_0(h)}^+)$. Then with the use of estimate \eqref{3.15} we first extend this solution to the whole time segment $[0,T]$ and then similarly to the end of the proof of the previous theorem pass to the limit as $h\to +\infty$ and construct the desired solution. Note that here due to \eqref{3.21} $g(u)\phi_x \in L_1(\Pi_T^+)$ for all values of $p$ without any additional assumptions on the weight function $\psi$.
\end{proof}

\section{Uniqueness and continuous dependence}\label{S4}

\begin{theorem}\label{T4.1}
Let $p\in [0,2]$ in \eqref{1.12}, $\psi(x)$ be an admissible weight function, such that $\psi'(x)$ is also an admissible weight function and inequality \eqref{1.13} be verified. Then for any $T>0$ and $M>0$ there exists a constant $c=c(T,M)$, such that for any two weak solutions $u(t,x,y)$ and $\widetilde u(t,x,y)$ to problem \eqref{1.1}--\eqref{1.4}, satisfying $\|u\|_{X_w^{\psi(x)}(\Pi_T^+)}, \|\widetilde u\|_{X_w^{\psi(x)}(\Pi_T^+)} \leq M$, with corresponding data $u_0, \widetilde u_0\in L_{2,+}^{\psi(x)}$, $f, \widetilde f\in L_1(0,T;L_{2,+}^{\psi(x)})$ the following inequality holds:
\begin{equation}\label{4.1}
\|u -\widetilde u\|_{X_w^{\psi(x)}(\Pi_T^+)} \leq c\bigl( \|u_0 - \widetilde u_0\|_{L_{2,+}^{\psi(x)}} +
\|f-\widetilde f\|_{L_1(0,T;L_{2,+}^{\psi(x)})}\bigr).
\end{equation} 
\end{theorem}

\begin{proof}
Let $w\equiv u -\widetilde u$, $w_0\equiv u_0 - \widetilde u_0$, $F\equiv f- \widetilde f$.
Apply for the function $w$ Lemma~\ref{L2.7}, where $f_1\equiv 0$. Note that inequality \eqref{1.13} implies that $(\psi/\psi')^{1/4} \leq c(\psi')^{3p/8}\psi^{p/8}$ and, therefore,
\begin{multline}\label{4.2}
\Bigl(\iint |u|^{2p} u_x^2 \psi\, dxdy\Bigr)^{1/2}   \leq
\bigl\| |u|^p (\psi/\psi')^{1/4}\bigr\|_{L_{\infty,+}} \Bigl(\iint u_x^{2} (\psi'\psi)^{1/2}\,dxdy\Bigr)^{1/2} \\ \leq 
c \bigl\| u (\psi')^{3/8} \psi^{1/8}\bigr\|^p_{L_{\infty,+}} \bigl\|u_x (\psi')^{1/4}\psi^{1/4}\bigr\|_{L_{2,+}} \\ \leq
c_1 \Bigl( \iint (u_{xx}^2 +u_y^2 +u^2)\psi' \,dxdy\Bigr)^{3p/8+ 1/4} \Bigl(\iint u^2 \psi\, dxdy\Bigr)^{p/8+1/4},
\end{multline}
and so $g'(u)u_x \in L_1(0,T;L_{2,+}^{\psi(x)})$ since $p\leq 2$.

As a result, we derive from \eqref{2.30} that for $t\in (0,T]$
\begin{multline}\label{4.3}
\iint w^2\psi\,dxdy +\int_0^t\!\! \iint (4w_{xx}^2 +3w_x^2 +w_y^2)\psi'\,dxdyd\tau \leq \iint w_0^2\psi\,dxdy  \\ 
+c\int_0^t\!\! \iint w^2\psi\,dxdyd\tau + 2\int_0^t \!\!\iint\bigl(F-(g'(u)u_x-g'(\widetilde u) \widetilde u_x)\bigr)w\psi\, dxdyd\tau.
\end{multline}
Here
\begin{multline}\label{4.4}
2\Bigl|\iint \bigl(g'(u)u_x-g'(\widetilde u) \widetilde u_x\bigr) w\psi\,dxdy \Bigr|  = 
2\Bigl| \iint \bigl(g(u) - g(\widetilde u)\bigr) (w\psi)_x \,dxdy \Bigr|\\ \leq 
c\iint (|u|^p+|\widetilde u|^p)|w (w\psi)_x|\,dxdy,
\end{multline}
where similarly to \eqref{4.2}
\begin{multline}\label{4.5}
\iint |u|^p |ww_x| \psi\,dxdy  \leq \bigl\| |u|^p (\psi/\psi')^{1/4}\bigr\|_{L_{\infty,+}} 
\Bigl(\iint w_x^2(\psi')^{1/2}\psi^{1/2} \,dxdy \Bigr)^{1/2}
\Bigl(\iint w^2 \psi \,dxdy \Bigr)^{1/2} \\ \leq
c\Bigl(\iint (u_{xx}^2 +u_y^2 +u^2)\psi'\, dxdy\Bigr)^{3p/8} \Bigl(\iint u^2\psi \,dxdy\Bigr)^{p/8} \\ \times
\Bigl(\iint (w_{xx}^2 +w_y^2 +w^2)\psi'\, dxdy\Bigr)^{1/4} \Bigl(\iint w^2\psi \,dxdy\Bigr)^{3/4} \\ \leq
\varepsilon \iint (w_{xx}^2 +w_y^2 +w^2)\psi'\, dxdy   +c(\varepsilon) \Bigl( \iint (u_{xx}^2 + u_y^2 +u^2)\psi' \,dxdy\Bigr)^{p/2}
\iint w^2\psi \,dxdy,
\end{multline}
where $\varepsilon>0$ can be chosen arbitrarily small. Then inequalities \eqref{4.3}, \eqref{4.5} provide the desired result.
\end{proof}

\begin{remark}\label{R4.1}
Theorems~\ref{T1.1} and \ref{T4.1} show that under the hypothesis of Theorem~\ref{T1.1} problem \eqref{1.1}--\eqref{1.4} is globally well-posed in the space $X_w^{\psi(x)}(\Pi_T^+)$.
\end{remark}

The next theorem, in particular, provides the uniqueness part of Theorem~\ref{T1.2}.

\begin{theorem}\label{T4.2}
Let the function $g\in C^2(\mathbb R)$ verifies condition \eqref{1.14}. Let $\psi(x)$ be an admissible weight function, such that $\psi'(x)$ is also an admissible weight function and condition \eqref{1.15} holds. Then for any $T>0$ and $M>0$ there exists a constant $c=c(T,M)$, such that for any two strong solutions $u(t,x,y)$ and $\widetilde u(t,x,y)$ to problem \eqref{1.1}--\eqref{1.4}, satisfying $\|u\|_{X_w^{1,\psi(x)}(\Pi_T^+)}, \|\widetilde u\|_{X_w^{1,\psi(x)}(\Pi_T^+)} \leq M$, with the corresponding data $u_0, \widetilde u_0\in L_{2,+}^{\psi(x)}$, $f, \widetilde f\in L_1(0,T;L_{2,+}^{\psi(x)})$ inequality \eqref{4.1} holds.
\end{theorem}

\begin{proof}
The proof mostly repeats the proof of Theorem~\ref{T4.1}. Note that here obviously $g'(u)u_x, g'(\widetilde u)\widetilde u_x \in L_\infty(0,T:L_{2,+}^{\psi(x)})$, therefore, equality \eqref{4.3} holds. The difference is related only to the nonlinear term. In comparison with \eqref{4.4} we estimate it in the following way: since 
\begin{equation}\label{4.6}
g'(u)u_x-g'(\widetilde u) \widetilde u_x = \bigl(g'(u) - g'(\widetilde u)\bigr)u_x + g'(\widetilde u) w_x,
\end{equation}
\begin{multline}\label{4.7}
2\Bigl|\iint \bigl(g'(u)u_x-g'(\widetilde u) \widetilde u_x\bigr) w\psi\,dxdy \Bigr|  = 
\Bigl| 2\iint \bigl(g'(u) - g'(\widetilde u)\bigr)u_x w\psi \,dxdy - \iint g''(\widetilde u) u_x w^2 \psi\,dxdy  \\ -\iint g'(\widetilde u) w^2\psi' \,dxdy \Bigr|  \leq 
c\iint \bigl(|u|^q+|\widetilde u|^q\bigr)\bigl(|u_x| +|\widetilde u_x|\bigr)) w^2\psi\,dxdy +c \iint w^2 \psi \,dxdy.
\end{multline}
Since by virtue of \eqref{1.15} $\psi \leq c \psi^{(1+q)/2} (\psi')^{1/8} \psi^{7/8}$
\begin{multline}\label{4.8}
\iint |u|^q |u_x| w^2 \psi\,dxdy  \leq c\iint |u|^q \psi^{q/2} \cdot |u_x| \psi^{1/2} \cdot w^2 (\psi')^{1/8} \psi^{7/8} \,dxdy \\ \leq 
c\|u \psi^{1/2}\|^q_{L_{\infty,+}} \|u_x \psi^{1/2}\|_{L_{6,+}} \|w (\psi')^{1/16} \psi^{7/16} \|_{L_{12/5,+}}^2 \\ \leq
c_1 \|u\|^{q+1}_{\widetilde H_+^{1,\psi(x)}} \Bigl( \iint (w_{xx}^2 +w_y^2 +w^2)\psi' \,dxdy\Bigr)^{1/8} \Bigl(\iint w^2\psi \,dxdy\Bigr)^{7/8},
\end{multline}
the desired result succeeds from \eqref{4.3}, \eqref{4.7}.
\end{proof}

\begin{theorem}\label{T4.3}
Let the function $g\in C^2(\mathbb R)$ verifies condition \eqref{1.14}. Let $\psi(x)$ be an admissible weight function, such that $\psi'(x)$ is also an admissible weight function and
for certain positive constant $c_0$
\begin{equation}\label{4.9}
\psi'(x) \psi^{q}(x) \geq c_0\quad \forall x\geq 0.
\end{equation}
Then for any $T>0$ and $M>0$ there exists a constant $c=c(T,M)$, such that for any two strong solutions $u(t,x,y)$ and $\widetilde u(t,x,y)$ to problem \eqref{1.1}--\eqref{1.4}, satisfying $\|u\|_{X_w^{1,\psi(x)}(\Pi_T^+)}, \|\widetilde u\|_{X_w^{1,\psi(x)}(\Pi_T^+)} \leq M$, with corresponding data $u_0, \widetilde u_0\in \widetilde H_+^{1,\psi(x)}$, $f, \widetilde f\in L_2(0,T;\widetilde H_+^{1,\psi(x)})$, $u_0(0,y) = \widetilde u_0(0,y)\equiv 0$, the following inequality holds:
\begin{equation}\label{4.10}
\|u -\widetilde u\|_{X_w^{1,\psi(x)}(\Pi_T^+)} \leq c\bigl( \|u_0 - \widetilde u_0\|_{H_+^{1,\psi(x)}}  +
\|f-\widetilde f\|_{L_2(0,T;H_+^{1,\psi(x))})}\bigr).
\end{equation} 
\end{theorem}

\begin{proof}
First of all note that the hypothesis of Theorem~\ref{T4.2} is satisfied and, consequently, inequality \eqref{4.1} holds.

Let $g'_1(u) \equiv g'(u) - g'(0)$, then according to \eqref{1.14} 
\begin{equation}\label{4.11}
|g'_1(u)| \leq c |u|^{q+1}.
\end{equation}
Adjoin the term $g'(0)u_x$ to the linear term $bu_x$ and consider an equation of \eqref{1.1} type, where $g'$ is substituted by $g'_1$. Condition \eqref{4.9} implies that 
\begin{equation}\label{4.12}
\frac{\psi^{2}(x)}{\psi'(x)} \leq c \psi^{q+2}(x).
\end{equation}
In particular, it means that $g'_1(u)u_x, g'_1(\widetilde u)\widetilde u_x \in L_\infty(0,T;L_{2,+}^{\psi^2(x)/\psi'(x)})$. Write corresponding equality \eqref{2.35} for $w\equiv u-\widetilde u$ and $f_1 \equiv g'_1(u)u_x- g'_1(\widetilde u)\widetilde u_x$, then 
\begin{multline}\label{4.13}
\iint (w_{xx}^2 +w_x^2 +w_y^2) \psi \,dxdy  + \int_0^t \!\! \iint (4w_{xxxx}^2 +5w_{xxy}^2 +u_{yy}^2 +7u_{xxx}^2 +3u_{xy}^2)\psi' \,dxdy d\tau \\ \leq
\iint (w_{0xx}^2 +w_{0x}^2 +w_{0y}^2)\psi\, dxdy +
c\int_0^t \!\! \iint (w_{xx}^2 +w_x^2 +w_y^2 +w^2)\psi \, dxdy d\tau \\ +
c\int_0^t \!\! \iint (F_{xx}^2 +F_y^2 +F^2)\psi \,dxdy d\tau +
c\int_0^t \!\! \iint \bigl( g'_1(u)u_x- g'_1(\widetilde u)\widetilde u_x\bigr)^2 \frac{\psi^2}{\psi'} \,dxdy d\tau.
\end{multline} 
To estimate the integral with the nonlinear term apply \eqref{4.11}, \eqref{4.12} and the corresponding analogue of \eqref{4.6}, then
\begin{multline*}
\iint \bigl( g'_1(u)u_x- g'_1(\widetilde u)\widetilde u_x\bigr)^2 \frac{\psi^2}{\psi'} \,dxdy \leq c\iint \bigl(|u|^{2q}+|\widetilde u|^{2q}\bigr) u_x^2 w^2 \psi^{q+2} \,dxdy \\+
c\iint |\widetilde u|^{2q+2} w_x^2 \psi^{q+2} \, dxdy,
\end{multline*} 
where
\begin{multline*}
\iint |u|^{2q} u_x^2 w^2 \psi^{q+2} \,dxdy \leq \|u \psi^{1/2}\|^{2q}_{L_{\infty,+}} \|u_x \psi^{1/2}\|^2_{L_{6,+}} \|w\psi^{1/2}\|_{L_{3,+}}^2 \\ \leq 
c\|u\|^{2q+2}_{\widetilde H_+^{1,\psi(x)}} \iint (w_{xx}^2 +w_y^2 +w^2) \psi\, dxdy,
\end{multline*}
$$
\iint |\widetilde u|^{2q+2} w_x^2 \psi^{q+2} \, dxdy \leq \|u \psi^{1/2}\|^{2q+2}_{L_{\infty,+}} \iint w_x^2\psi \,dxdy.
$$
As a result, the statement of the theorem follows from inequality \eqref{4.13}.
\end{proof}

\begin{remark}\label{R4.2}
Theorems~\ref{T1.2}, \ref{T4.2} and \ref{T4.3} show that under the hypothesis of Theorem~\ref{T1.2} and additional assumption  \eqref{4.9}  problem \eqref{1.1}--\eqref{1.4} is globally well-posed in the space $X_w^{1,\psi(x)}(\Pi_T^+)$. This additional assumption holds for any exponential weight $e^{2\alpha x}$, $\alpha>0$, and for the power weight $(1+x)^{2\alpha}$ if $\alpha\geq 1/(2(q+1))$.
\end{remark}

\section{Large-time decay of solutions}\label{S5}

\begin{proof}[Proof of Theorem~\ref{T1.3}]
Let $\psi(x)\equiv e^{2\alpha x}$ for certain $\alpha \in (0,\alpha_0]$, where $\alpha_0$ will be specified later, $u_0\in L_{2,+}^{\psi(x)}$, $f\equiv 0$. Consider the unique solution to problem \eqref{1.1}--\eqref{1.4} (in the cases a) and c)) $u\in X_w^{\psi(x)}(\Pi_T^+)\ \forall T$. Note that according to \eqref{4.2} $g'(u)u_x\in L_1(0,T;L_{2,+}^{\psi(x)})$.

Apply Lemma~\ref{L2.7}, where $f_0\equiv g'(u)u_x$, $f_1\equiv 0$, then equality \eqref{3.6} for $\rho\equiv 1$ provides that 
\begin{equation}\label{5.1}
\|u(t,\cdot,\cdot)\|_{L_{2,+}}\leq \|u_0\|_{L_{2,+}} \quad \forall t\geq 0.
\end{equation} 
Next, equality \eqref{3.6} for $\rho \equiv \psi$ implies that
\begin{multline}\label{5.2}
\frac{d}{dt} \iint u^2\psi\,dxdy +
2\alpha \iint (5u_{xx}^2 +u_x^2 +u_y^2 )\psi\,dxdy 
+2\alpha(2-20\alpha^2) \iint u_x^2\psi\,dxdy \\
+2\alpha(16\alpha^4 -4\alpha^2 -b) \iint u^2\psi\,dxdy + \int_0^L \mu_2^2 \,dy
= 2\alpha \iint \bigl(g'(u)u\bigr)^*\psi\,dxdy.
\end{multline} 
With the use of inequalities \eqref{3.8} (without $h$) and \eqref{3.9} we derive that uniformly with respect to $L$ for certain constant $c^*$, depending on the properties of the function $g$,
\begin{multline}\label{5.3}
2\iint \bigl(g'(u)u\bigr)^*\psi\,dxdy   \leq
c\Bigl(\iint (u_{xx}^2 +u_y^2+u^2)\psi\, dxdy\Bigr)^{3p/8} \Bigl(\iint u^2\psi\, dxdy\Bigr)^{(8-3p)/8} \|u_0\|^p_{L_{2,+}} \\ \leq
\frac12  \iint (u_{xx}^2 +u_y^2)\psi\, dxdy +c^*\bigl(\|u_0\|^{8p/(8-3p)}_{L_{2,+}} +\|u_0\|^p_{L_{2,+}}\bigr) \iint u^2\psi\, dxdy.
\end{multline}
Inequality \eqref{2.8} yields that 
\begin{equation}\label{5.4}
\frac12 \iint u_y^2\psi\,dxdy \geq \frac{\pi^2}{2\varkappa\L^2} \iint u^2\psi\,dxdy.
\end{equation}
Combining \eqref{5.2}--\eqref{5.4} we find that uniformly with respect to $\alpha\in (0,\sqrt{1/10}]$ and $L$
\begin{multline}\label{5.5}
\frac{d}{dt}\iint u^2\psi\,dxdy + \alpha \iint \bigl(u_{xx}^2 +u_x^2+u_y^2\bigr)\psi\,dxdy +\int_0^L \mu_2^2\,dy
\\ +\alpha\Bigl(\frac{\pi^2}{2\varkappa\L^2}-2b-8\alpha^2- 32\alpha^4 -c^* \bigl(\|u_0\|^{8p/(8-3p)}_{L_{2,+}} +\|u_0\|^p_{L_{2,+}}\bigr)\Bigr) \iint u^2\psi\,dxdy
\leq 0.
\end{multline}
Choose $\displaystyle L_0= \frac{\pi}{\sqrt{20\varkappa b}}$ if $ b>0$, $\alpha_0 \in (0,\sqrt{1/10}]$, satisfying an inequality $\displaystyle 8\alpha_0^2 + 32\alpha_0^4 \leq  \frac{\pi^2}{10\varkappa L^2}$, $\epsilon_0>0$ satisfying an inequality $\displaystyle c^*(\epsilon_0^{8p/(8-3p)}+\epsilon_0^p) \leq \frac{\pi^2}{10\varkappa L}$, $\beta = \displaystyle \frac{\pi^2}{10\varkappa L^2}$. Then it follows from \eqref{5.5} that
\begin{equation}\label{5.6}
\frac{d}{dt}\iint u^2\psi\,dxdy + \int_0^L \mu_2^2\,dy +
\alpha \iint (u_{xx}^2 +u_x^2+u_y^2 +\beta u^2)\psi\,dxdy 
+\alpha\beta\iint u^2\psi\,dxdy
\leq 0,
\end{equation}
whence \eqref{1.16} follows.
\end{proof}

\begin{proof}[Proof of Theorem~\ref{T1.4}]
Let the values $L_0$, $\alpha_0$, $\epsilon_0$, $\beta$ be the same as as in the proof of the previous theorem, $\psi(x)\equiv e^{2\alpha x}$ for certain $\alpha \in (0,\alpha_0]$,  $u_0\in \widetilde H_+^{1,\psi(x)}$, $u_0(0,y)\equiv 0$, $\|u_0\|_{L_2,+} \leq \epsilon_0$, $f\equiv 0$.  Consider the unique solution to problem \eqref{1.1}--\eqref{1.4} (in the cases a) and c)) $u\in X_w^{1,\psi(x)}(\Pi_T^+)\ \forall T$. Since $g'(u)u_x \in L_\infty(0,T;L_{2,+}^{\psi(x)})$, all the proof of Theorem~\ref{T1.3} can be repeated and inequality \eqref{5.6} obtained. In particular, besides \eqref{1.16} it follows from \eqref{5.6} that
\begin{equation}\label{5.7}
\int_0^{+\infty}\ e^{\alpha\beta \tau} \Bigl[\int_0^L u_{xx}^2\big|_{x=0}\,dy +
\alpha \iint (u_{xx}^2 +u_x^2+u_y^2 +\beta u^2)\psi\,dxdy\Bigr]\,d\tau \leq \|u_0\|^2_{L_{2,+}^{\psi(x)}}.
\end{equation}

Next, similarly to \eqref{3.17} we derive from \eqref{2.35} and \eqref{2.39} that in the case $\rho(x)\equiv 1$
\begin{equation}\label{5.8}
\frac{d}{dt}\iint(u_{xx}^2+u_x^2+u_y^2 -2g^*(u)) \,dxdy \leq c\int_0^L u^2_{xx}\big|_{x=0} \, dy, 
\end{equation}
whence with the use of \eqref{3.18} and \eqref{5.7} follows that uniformly with respect to $t\geq 0$
$$
\|u_{xx}\|_{L_{2,+}} + \|u_{y}\|_{L_{2,+}} \leq c
$$
and, in particular,
\begin{equation}\label{5.9}
\|u\|_{L_\infty(\Pi^+_\infty)} \leq c.
\end{equation}
Returning to equalities \eqref{2.35} and \eqref{2.39} in the case $\rho(x)\equiv \psi(x)$ we derive that
\begin{multline}\label{5.10}
\frac{d}{dt}\iint(u_{xx}^2+u_x^2+u_y^2 -2g^*(u))\psi \,dxdy  
+2\alpha\iint(4u_{xxxx}^2+5u^2_{xxy}+u^2_{yy} +8u_{xxx}^2 +4u_{xy}^2)\psi \,dxdy \\ \leq
c\alpha\iint (u_{xx}^2+u_x^2 +u_y^2)\psi\, dxdy 
-4b\alpha\iint g^*(u)\psi\, dxdy -2\alpha \iint \bigl(g'(u)g(u)\bigr)^*\psi \,dxdy
+c\int_0^L u_{xx}^2\big|_{x=0} \,dy \\
+2\alpha\iint g(u)\big[(6u_{xxxx} +(8\alpha^2-4)u_{xx} -2u_{yy} +16\alpha u_{xxx} -4\alpha u_x\bigr]\psi \,dxdy,
\end{multline}
whence with the use of \eqref{5.7} and \eqref{5.9} inequality \eqref{1.17} succeeds.
\end{proof}

\appendix
\section{}\label{A1}

For certain interval $I\subset \mathbb R$ and natural $m$ consider the usual Sobolev space $H^m_0(I)$ and its adjoint one $H^{-m}(I) = \bigl(H^m_0(I)\bigr)^*$. Then
\begin{equation}\label{A.1}
\|f\|_{H^{-m}(I)} = \sup\limits_{\varphi \in H_0^m(I), \varphi\not\equiv 0} \frac{|\langle f,\varphi\rangle|}{\|\varphi\|_{H^m(I)}}.
\end{equation}
The structure of the space $H^{-m}(I)$ is well-known: $f \in H^{-m}(I)$ iff $f(x)= \sum\limits_{j=0}^m f_j^{(j)}(x), \ f_j \in L_2(I)$.  Moreover, the equivalent norm in this space is the following:
\begin{equation}\label{A.2}
\|f\|_{H^{-m}(I)} = \Bigl(\sum\limits_{j=0}^m \|f_j\|^2_{L_2(I)}\Bigr)^{1/2}.
\end{equation}

\begin{lemma}\label{LA.1}
Let $k\geq n+m$ for certain natural $n$ and $m$. If $f\in L_2(\mathbb R)$ and $f^{(k)} \in H^{-m}(\mathbb R)$, then $f^{(n)} \in L_2(\mathbb R)$ and for certain constant $c=c(k,m,n)$
\begin{equation}\label{A.3}
\|f^{(n)}\|_{L_2(\mathbb R)} \leq c\bigl(\|f^{(k)}\|_{H^{-m}(\mathbb R)} + \|f\|_{L_2(\mathbb R)}\bigr).
\end{equation}
\end{lemma}

\begin{proof}
In the case of $I= \mathbb R$ there is another definition of the Sobolev spaces for $s\in \mathbb R$
$$
H^s(\mathbb R) = \bigl\{ f\in \EuScript S'(\mathbb R): (1+\xi^2)^{s/2} \widehat f(\xi) \in L_2(\mathbb R)\bigr\} 
$$
with the equivalent norm
$$
\|f\|_{H^s(\mathbb R) } = \frac1{\sqrt{2\pi}} \|(1+\xi^2)^{s/2} \widehat f(\xi)\|_{L_2(\mathbb R)}.
$$
Since $\widehat{f^{(l)}}(\xi) = (i\xi)^l \widehat f(\xi)$, estimate \eqref{A.3} follows from the obvious inequality
$$
|\xi|^{n} \leq c\bigl( |\xi|^k (1+|\xi|^2)^{-m/2} +1\bigr).
$$
\end{proof}

\begin{lemma}\label{LA.2}
Let $k\geq n+m$ for certain natural $n$ and $m$, $I$ be a bounded interval. If $f\in L_2(I)$ and $f^{(k)} \in H^{-m}(I)$, then $f^{(n)} \in L_2(I)$ and for certain constant $c=c(k,m,n,|I|)$
\begin{equation}\label{A.4}
\|f^{(n)}\|_{L_2(I)} \leq c\bigl(\|f^{(k)}\|_{H^{-m}(I)} + \|f\|_{L_2(I)}\bigr).
\end{equation}
\end{lemma}

\begin{proof}
Without loss of generality assume for simplicity that $I=(0,1)$. Let $\varphi \in L_2(0,1)$. Set $\varphi_0(x) \equiv \varphi(x)$ and for any natural $j$
\begin{equation}\label{A.5}
a_j = \int_0^1 \varphi_{j-1}(x')\,dx',\quad \varphi_j(x) \equiv \int_0^x \varphi_{j-1}(x') \,dx' - a_j \eta(x).
\end{equation}
Then $\varphi_j \in H_0^j(0,1)$ and
\begin{equation}\label{A.6}
\|\varphi_j\|_{H^j(0,1)}, |a_j| \leq c(j)\|\varphi\|_{L_2(0,1)},\quad
\varphi(x) = \varphi_j^{(j)}(x) + a_j\eta^{(j)}(x)+\dots a_1 \eta'(x).
\end{equation}
Here for $j= k-n \geq m$
\begin{multline*}
|\langle f^{(n)},\varphi_j^{(j)}\rangle| = |\langle f^{(k)}, \varphi_j\rangle| \leq \|f^{(k)}\|_{H^{-m}(0,1)} \|\varphi_j\|_{H^m(0,1)} \\ \leq
\|f^{(k)}\|_{H^{-m}(0,1)} \|\varphi_j\|_{H^j(0,1)} \leq c \|f^{(k)}\|_{H^{-m}(0,1)} \|\varphi\|_{L_2(0,1)},
\end{multline*}
for any $l$
$$
|\langle f^{(n)}, a_l \eta^{(l)}\rangle| \leq |a_l \langle f, \eta^{(l+n)}\rangle| \leq c(l,n)\|f\|_{L_2(0,1)} \|\varphi\|_{L_2(0,1)}
$$
and \eqref{A.6} implies \eqref{A.4}.
\end{proof}

\begin{lemma}\label{LA.3}
Let $k\geq  n+m$ for certain natural $n$ and $m$. If $f\in H^{n-1}(\mathbb R_+)$ and $f^{(k)} \in H^{-m}(\mathbb R_+)$, then $f^{(n)} \in L_2(\mathbb R_+)$ and for certain constant $c=c(k,m,n)$
\begin{equation}\label{A.7}
\|f^{(n)}\|_{L_2(\mathbb R_+)} \leq c\bigl(\|f^{(k)}\|_{H^{-m}(\mathbb R_+)} + \|f\|_{H^{n-1}(\mathbb R_+)}\bigr).
\end{equation}
\end{lemma}

\begin{proof}
First consider the case $k=n+m$. Let $\eta_0(x)\equiv \eta(x)$, $\eta_1(x) \equiv \eta(1-x)$, $f_0(x) \equiv f(x)\eta_0(x)$, $f_1(x) \equiv f(x)\eta_1(x)$ (then $f(x) \equiv f_0(x) +f_1(x)$).

Extend the function $f_0$ by zero to the whole real line. Let $\varphi\in H^m(\mathbb R)$, then
\begin{multline}\label{A.8}
\langle f_0^{(k)}, \varphi\rangle = (-1)^m \langle f_0^{(n)}, \varphi^{(m)}\rangle = (-1)^m \sum\limits_{j=0}^n \binom{n}{j} \langle f^{(j)}, \eta_0^{(n-j)}\varphi^{(m)} \rangle \\ =
(-1)^m \bigl[\langle f^{(n)}, (\eta_0\varphi)^{(m)}\rangle - \sum\limits_{l=0}^{m-1} \binom{m}{l}\langle f^{(n)}, \eta_0^{(m-l)} \varphi^{(l)} \rangle 
+ \sum\limits_{j=0}^{n-1} \binom{n}{j} \langle f^{(j)}, \eta_0^{(n-j)}\varphi^{(m)}\bigr]  \\=
\langle f^{(k)}, \eta_0\varphi\rangle +(-1)^m \bigl[\sum\limits_{l=0}^{m-1} \binom{m}{l}\langle f^{(n-1)}, (\eta_0^{(m-l)} \varphi^{(l)})' \rangle
+ \sum\limits_{j=0}^{n-1} \binom{n}{j} \langle f^{(j)}, \eta_0^{(n-j)}\varphi^{(m)}\bigr] \\ \leq
c\bigl(\|f^{(k)}\|_{H^{-m}(\mathbb R_+)} + \|f\|_{H^{n-1}(\mathbb R_+)}\bigr) \|\varphi\|_{H^m(\mathbb R)},
\end{multline}
whence follows that
\begin{equation}\label{A.9}
\|f_0^{(k)}\|_{H^{-m}(\mathbb R)} \leq \bigl(\|f^{(k)}\|_{H^{-m}(\mathbb R_+)} + \|f\|_{H^{n-1}(\mathbb R_+)}\bigr),
\end{equation}
and then with the use of \eqref{A.3}
that
\begin{equation}\label{A.10}
\|f_0^{(n)}\|_{L_2(\mathbb R_+)} \leq \bigl(\|f^{(k)}\|_{H^{-m}(\mathbb R_+)} + \|f\|_{H^{n-1}(\mathbb R_+)}\bigr).
\end{equation}

For the function $f_1$ the same argument can be repeated for a function $\varphi\in H^m_0(0,1)$ with the natural substitution of $\eta_0$, $\mathbb R$ by $\eta_1$, $(0,1)$. Then similarly to \eqref{A.8}, \eqref{A.9}
$$
\|f_1^{(k)}\|_{H^{-m}(0,1)} \leq  \bigl(\|f^{(k)}\|_{H^{-m}(\mathbb R_+)} + \|f\|_{H^{n-1}(\mathbb R_+)}\bigr),
$$
whence with use of \eqref{A.4} follows that
$$
\|f_1^{(n)}\|_{L_2(\mathbb R_+)} \leq \bigl(\|f^{(k)}\|_{H^{-m}(\mathbb R_+)} + \|f\|_{H^{n-1}(\mathbb R_+)}\bigr).
$$
To finish the proof in the general case it is suffice to note that $H^{-m}(\mathbb R_+) \subset H^{n-k}(\mathbb R_+)$.
\end{proof}

Consider also one analogous result for weighted spaces, although it is not used in the present paper. For an admissible weight function $\psi(x)$ let $H^{k,\psi(x)}(\mathbb R_+)$ be a space of functions $\varphi(x)$ verifying $\varphi\psi^{1/2} \in H^k(\mathbb R_+)$, $L_2^{\psi(x)}(\mathbb R_+) = H^{0,\psi(x)}(\mathbb R_+)$. It is easy to see that $H^{k,\psi(x)}(\mathbb R_+) = \{\varphi(x): \varphi^{(j)}\psi^{1/2} \in L_2(\mathbb R_+), \ j=0,\dots,k\}$. In accordance to \eqref{A.2} define $H^{-k,\psi(x)} = \{ f(x) = \sum\limits_{j=0}^k f_j^{(j)}(x): f_j \in L_2^{\psi(x)}(\mathbb R_+)\ \forall j\}$ equipped with the natural norm.

\begin{lemma}\label{LA.4}
$f\in H^{-k,\psi(x)}(\mathbb R_+)$ iff $f\psi^{1/2} \in H^{-k}(\mathbb R_+)$ with the corresponding equivalence of the norms.
\end{lemma}

\begin{proof}
Let $f\in H^{-k,\psi(x)}(\mathbb R_+)$, then
$$
f\psi^{1/2} = \sum\limits_{j=0}^k f_j^{(j)} \psi^{1/2} = \sum\limits_{j=0}^k \sum\limits_{l=0}^j c_{jl} \bigl(f_j (\psi^{1/2})^{(j-l)}\bigr)^{(l)}  =
\sum\limits_{l=0}^k \Bigl(\sum\limits_{j=l}^k c_{jl} f_j (\psi^{1/2})^{(j-l)} \Bigr)^{(l)} \equiv \sum\limits_{l=0}^k \widetilde f_l^{(l)},
$$
where $\widetilde f_l^{(l)} \in L_2(\mathbb R_+)$ with the corresponding estimate on the norm.

Conversely, if $f\psi^{1/2} \in H^{-k}(\mathbb R_+)$, then
\begin{multline*}
f = (f\psi^{1/2})\psi^{-1/2} = \sum\limits_{j=0}^k f_j^{(j)} \psi^{-1/2} = \sum\limits_{j=0}^k \sum\limits_{l=0}^j c_{jl} \bigl(f_j (\psi^{-1/2})^{(j-l)}\bigr)^{(l)} \\ =
\sum\limits_{l=0}^k \Bigl(\sum\limits_{j=l}^k c_{jl} f_j (\psi^{-1/2})^{(j-l)} \Bigr)^{(l)} \equiv \sum\limits_{l=0}^k \widetilde f_l^{(l)},
\end{multline*}
where $\widetilde f_l^{(l)} \in L_2^{\psi(x)}(\mathbb R_+)$ with the corresponding estimate on the norm.
\end{proof}

\begin{lemma}\label{LA.5}
If $f\in H^{n-1,\psi(x)}(\mathbb R_+)$ and $f^{(n+1)} \in H^{-1,\psi(x)}(\mathbb R_+)$ for certain admissible weight function $\psi(x)$ and natural $n$, then $f^{(n)} \in L_2^{\psi(x)}(\mathbb R_+)$ and for certain constant $c=c(n,\psi)$
\begin{equation}\label{A.11}
\|f^{(n)}\|_{L_2^{\psi(x)}(\mathbb R_+)} \leq c\bigl(\|f^{(n+1)}\|_{H^{-1,\psi(x)}(\mathbb R_+)} + \|f\|_{H^{n-1,\psi(x)}(\mathbb R_+)}\bigr).
\end{equation}
\end{lemma}

\begin{proof}
By virtue of \eqref{A.7}
\begin{multline*}
\|f^{(n)}\|_{L_2^{\psi(x)}(\mathbb R_+)} \leq \|(f\psi^{1/2})^{(n)}\|_{L_2(\mathbb R_+)} + c\|f\|_{H^{n-1,\psi(x)}(\mathbb R_+)} \\ \leq
c_1 \bigl(\|(f\psi^{1/2})^{(n+1)}\|_{H^{-1}(\mathbb R_+)} + \|f\|_{H^{n-1,\psi(x)}(\mathbb R_+)} \bigr),
\end{multline*}
where
\begin{multline*}
(f\psi^{1/2})^{(n+1)} \\= f^{(n+1)}\psi^{1/2} + (n+1)\bigl[\bigl( f^{(n-1)}(\psi^{1/2})'\bigr)' - f^{(n-1)}(\psi^{1/2})''\bigr] + \sum\limits_{j=0}^{n-1} \binom{n+1}{j} f^{(j)} (\psi^{1/2})^{(n+1-j)}
\end{multline*}
 and according to \eqref{A.2}
$$
\|\bigl( f^{(n-1)}(\psi^{1/2})'\bigr)'\|_{H^{-1}(\mathbb R_+)} \leq c \|f^{(n-1)}\|_{L_2^{\psi(x)}}.
$$
\end{proof}

\begin{remark}\label{RA.1}
Inequality \eqref{A.11} was used in \cite{F18, F20, F21} in the case $n=2$ without the detailed proof.
\end{remark}

\subsection*{Conflict of Interests}
The author declares that there is no conflict of interests regarding the publication of this paper.



\begin{thebibliography}{99}

\bibitem{BIN} O.~V.~Besov, V.~P.~Il'in and S.~M.~Nikolskii, \textit{Integral Representation of Functions and Embedding Theorems}, J. Wiley, 1978.

\bibitem{BL97} H.~A.~Biagioni and F.~Linares, \textit{On the Benney--Lin and Kawahara equations}, J. Math. Anal. Appl. {\bf 211} (1997), 131--152.

\bibitem{CK} M.~Cavalcante and Ch.~Kwak, \textit{The initial-boundary value problem for the Kawahara equation on the half-line}, Nonlinear Differential Equ. Appl. {\bf 27} (2020), art.~45.

\bibitem{CG} W.~Chen and Z.~Guo, {\it Global well-posedness and $I$-method for the fifth-order Korteweg--de~Vries equation}, J. d'Anal. Math. {\bf 114} (2011), 121--156.

\bibitem{CFL} L.~Cossetti, L.~Fanelli and F.~Linares, {\it Uniqueness results for Zakharov--Kuznetsov equation}, Comm. Partial Differential Equ. {\bf 44} (2019), 504--544.

\bibitem{CDT} S.~B.~Cui, D.~G.~Deng and S.~P.~Tao, \textit{Global existence of solutions for the Cauchy problem of the Kawahara equation with $L_2$ initial data}, Acta Math. Sin. (Engl. Ser.) {\bf 22} (2006), 1457--1466.

\bibitem{DL} G.~G.~Doronin and N.~A.~Larkin, {\it Quarter-plane problem for the Kawahara equation}, Pac. J. Appl. Math. {\bf 1} (2008), 151--176.

\bibitem{EEA} S.~A.~Elwakil, E.~K.~El-Shewy and H.~G.~Abdelwahed, \textit{Solution of the perturbed Zakharov--Kuznetsov (ZK) equation describing electron-acoustic solitary waves in a magnetized plasma}, Chi. J. Phys. {\bf 49} (2011), 732-744.


\bibitem{F89} A.~V.~Faminskii, {\it The Cauchy problem for quasilinear equations of odd order}, Mat. Sb. {\bf 180} (1989), 1183--1210; English transl. in Math USSR Sb. {\bf 68} (1991), 31--59.

\bibitem{F12} A.~V.~Faminskii, {\it Weak solutions to initial-boundary-value problems for quasilinear evolution equations of an odd order}, Adv. Differential Equ. {\bf 17} (2012), 421--470.

\bibitem{F18} A.~V.~Faminskii, {\it Initial-boundary value problems in a half-strip for two-dimensional Zakharov--Kuznetsov equation}, Ann. Inst. H.~Poincar\'e (C) Analyse Non Lin\'eaire {\bf 35} (2018), 1235--1265.

\bibitem{F20} A.~V.~Faminskii, {\it Regular solutions to initial-boundary value problems  in a half-strip for two-dimensional Zakharov--Kuznetsov equation}, Nonlinear Anal. Real World Appl. {\bf 51} (2020), 102959.

\bibitem{F21} A.~V.~Faminskii, {\it Initial-boundary value problems on a half-strip for the modified Zakharov--Kuznetsov equation}, J. Evol. Equ. {\bf 21} (2021), 1263--1298.

\bibitem{FM21} A.~V.~Faminskii and E.~V.~Martynov, {\it Large-time decay of solutions of the damped Kawahara equation on the half-line}, Differential Equations on Manifolds and Mathematical Physics, V.~M.~Manuilov et al. eds. Trends in Mathematics, Birkh\"auser, 2021, 130--141.

\bibitem{FO14} A.~V.~Faminskii and M.~A.~Opritova, {\it On the initial value problem for the Kawahara equation}, J. Math. Sci. {\bf 201} (2014), 614--633.

\bibitem{FO15} A.~V.~Faminskii and M.~A.~Opritova, {\it On the initial-boundary-value problem in a half-strip for a generalized Kawahara equation}, J. Math. Sci. {\bf 206} (2015), 17--38.

\bibitem{GL} D.-A.~Geba and B.~Lin, {\it Unconditional well-posedness fir the Kawahara equation}, J. Math. Anal. Appl. {\bf 502} (2021), art.~125282.

\bibitem{KO} T.~Kakutani and H.~Ono, \textit{Weak nonlinear hydromagnetic waves in a cold collision-free plasma}, J. Phys. Soc. Japan \textbf{26} (1965), 1305--1318.

\bibitem{Kaw} T.~Kawahara, \textit{Oscillatory solitary waves in dispersive media}, J. Phys. Soc. Jpn. {\bf 33} (1972), 260--264.

\bibitem{KST}N.~Kishimoto, M.~Shan and Y.~Tsutsumi, {\it Localization estimate and global attractor for the damped and forced Zakharov--Kuznetsov equation}, Dyn. Partial Differential Equ. {\bf 16} (2019), 317--323.

\bibitem{KF} R.~V.~Kuvshinov and A.~V.~Faminskii, {\it Mixed problem for the Kawahara equation in a half-strip}, Differensial'nye Uravneniya {\bf 45} (2009), 391--402; English transl. in Differential Equ. {\bf 45} (2009), 404--415.

\bibitem{L14} N.~A.~Larkin, {\it The 2D Kawahara equation on a half-strip}, Appl. Math. Optim. {\bf 70} (2014), 443--468.

\bibitem{LS15} N.~A.~Larkin and M.~H.~Sim\~oes, {\it The Kawahara equation on bounded intervals and on a half-line}, Nonlinear Anal. {\bf 127} (2015), 397--412.

\bibitem{LS16} N.~A.~Larkin and M.~H.~Sim\~oes, {\it Global regular solutions for the 3D Kawahara equation posed on unbounded domains}, Z. Angew. Math. Phys. {\bf 67} (2016), 1--21.

\bibitem{MMPP} A.~Mendez, C.~Mu\~nos, F.~Poblete and J.~C.~Pozo, {\it On local energy decay for large solutions of the Zakharov--Kuznetsov equation}, Comm. Partial Differential Equ., {\bf 46} (2021), 1440--1487.

\bibitem{N} P.~I.~Naumkin {\it Time decay estimates for solutions of the Cauchy problem for the modified Kawahara equation}, Mat. Sb. {\bf 210} (2019), 72--108; English. transl. in Sbornik: Math. {\bf 210} (2019), 693--730.

\bibitem{SF} K.~Sangare and A.~V.~Faminskii, {\it Weak solutions of a mixed problem in a half-strip for a generalized Kawahara equation}, Mat. Zametki {\bf 85} (2009), 98--109; English transl. in Math. Notes {\bf 85} (2009), 90--100.

\bibitem{Sa} J.-C.~Saut, {\it Sur quelques g\'en\'eralizations de l'\'equation de Korteweg--de~Vries}, J. Math. Pures Appl. {\bf 58} (1979), 21--61.

\bibitem{Sh} M.~Shan, {\it Global well-posedness and global attractor for two-dimensional Zakharov--Kuznetsov equation},  Acta Math. Sin. (Engl. Ser.) {\bf 36} (2020), 969--1000.

\bibitem{Si} J. Simon, {\it Compact sets in the space $L^p(0,T;B)$}, Ann. Mat. Pura Appl. {\bf 146} (1987), 65--96.

\bibitem{SB} R.~Sipcic and D.~J.~Benney, {\it Lump interactions and collapse in the modified Zakharov--Kuznetsov equation}, Stud. Appl. Math. {\bf 105} (2000), 385--403.

\bibitem{TC} S.~P.~Tao and S.~B.~Cui, {\it Local and global existence of solutions to initial value problem of modified nonlinear Kawahara equation}, Acta Math. Sin. (Engl. Ser.) {\bf 21} (2005), 1035--1044.

\bibitem{ZK} V.~E.~Zakharov and E.~A.~Kuznetsov, {\it On three-dimensional solitons}, Zhurnal Eksp. Teoret. Fiz. {\bf 66} (1974), 594--597. English transl. in Soviet Phys. JETP {\bf 39} (1974), 285--288.

\bibitem{ZLSL} Z.~Zhang, Z.~Liu, M.~Sun and S.~Li, {\it Well-posedness and unique continuation property for the solutions to the generalized Kawahara equation below the energy space}, Appl. Anal. {\bf 97} (2018), 2655--2685.

\end{thebibliography}
\end{document}